\documentclass[12pt,oneside,a4paper,reqno]{amsart}

\usepackage[utf8]{inputenc}      
\usepackage[T1]{fontenc}         
\usepackage[english]{babel}  
\usepackage{geometry}            
\usepackage{amsmath, amssymb, amsthm} 
\usepackage[dvipsnames]{xcolor}
\definecolor{ourcolor}{RGB}{0,102,204}
\usepackage{hyperref}
 \hypersetup{
 colorlinks,
 citecolor=ourcolor,
 linkcolor=ourcolor,
 urlcolor=black}
\usepackage{fancyhdr}            
\usepackage{todonotes}
\usepackage{comment}
\theoremstyle{plain}
\newtheorem{teorema}{Teorema}[section]

\newtheorem{lemma}[teorema]{Lemma}

\theoremstyle{definition}

\theoremstyle{remark}

\newcommand{\R}{\mathbb{R}}
\newcommand{\N}{\mathbb{N}}

 \makeatletter
 \def\@seccntformat#1{\hspace*{0mm}%
  \protect\textup{\protect\@secnumfont
    \ifnum\pdfstrcmp{subsection}{#1}=0 \bfseries\fi
    \csname the#1\endcsname
    \protect\@secnumpunct
      }%
 }

\newcommand\II{I{\kern-1pt}I}

\newcommand\bdf{\begin{definition}}
\newcommand\bpr{\begin{proposition}}
\newcommand\brk{\begin{remark}}
\newcommand\blm{\begin{lemma}}
\newcommand\bexe{\begin{exercise}}
\newcommand\bexa{\begin{example}}
\newcommand\beqn{\begin{eqnarray*}}
\newcommand\edf{\end{definition}}
\newcommand\epr{\end{proposition}}
\newcommand\erk{\end{remark}}
\newcommand\elm{\end{lemma}}
\newcommand\eexe{\end{exercise}}
\newcommand\eexa{\end{example}}
\newcommand\eeqn{\end{eqnarray*}}

\numberwithin{equation}{section}
\theoremstyle{plain}
\newtheorem{theorem}{Theorem}[section]
\newtheorem{corollary}[theorem]{Corollary}

\newtheorem{proposition}[theorem]{Proposition}
\newtheorem{example}[theorem]{Example}
\theoremstyle{definition}

\newtheorem{definition}[theorem]{Definition}
\theoremstyle{remark}

\newtheorem{remark}[theorem]{Remark}

\newcommand{\veps}{\varepsilon}

\newcommand{\supp}{\mathop{\rm supp}\nolimits}   
\renewcommand{\d}{{\mathrm d}}

\newcommand{\restr}[1]{\lower3pt\hbox{$|_{#1}$}}

\DeclareUnicodeCharacter{043E}{\textendash}

\newcommand{\eps}{\varepsilon}
\newcommand{\nchi}{{\raise.3ex\hbox{$\chi$}}}

\begin{document}

 \title[Necessary and Sufficient Conditions for the MS Formula]{Necessary and Sufficient Conditions for the Maz\textquotesingle ya--Shaposhnikova Formula in (Fractional) Sobolev Spaces}

\author[E. Davoli]{Elisa Davoli}
\address{Institute of Analysis and Scientific Computing, TU Wien, Wiedner Hauptstraße 8-10, 1040 Vienna, Austria.}
\email{elisa.davoli@tuwien.ac.at}

\author[G. Di Fratta]{Giovanni Di Fratta}
\address{Dipartimento di Matematica e Applicazioni “R. Caccioppoli”, Università degli Studi di Napoli “Federico II”, Via Cintia, Complesso Monte S. Angelo, 80126 Naples, Italy.}
\email{giovanni.difratta@unina.it}

\author[R. Giorgio]
{Rossella Giorgio}
\address{Institute of Analysis and Scientific Computing, TU Wien, Wiedner Hauptstraße 8-10, 1040 Vienna, Austria and MedUni Wien, Währinger Gürtel 18-20, 1090 Vienna, Austria.}
\email{rossella.giorgio@student.tuwien.ac.at}

\author[A.~Pinamonti]{Andrea Pinamonti}
\address{Dipartimento di Matematica\\ Università di Trento\\
Via Sommarive, 14, 38123 Povo TN, Italy}
\email{andrea.pinamonti@unitn.it}

\begin{abstract}
We investigate the asymptotic behavior, as $\varepsilon
\rightarrow 0$, of nonlocal functionals
\[ \mathcal{F}_{\varepsilon} (u) = \iint_{\R^N \times \R^N} \rho_{\varepsilon}
   (y - x) \, | u (x) - u (y) |^p \hspace{0.17em}  \mathrm{d} x \,
   \mathrm{d} y, \quad u \in L^p ( \R^N ), \quad 1 \leqslant p <
   \infty, \]
associated with a general family of non-negative measurable kernels $\{ \rho_{\varepsilon}
\}_{\varepsilon > 0}$. Our primary aim is to single out the weakest
moment-type assumptions on the family of kernels $\{ \rho_{\varepsilon}
\}_{\varepsilon > 0}$ that are necessary and sufficient for the pointwise
convergence
\[ \underset{\varepsilon \rightarrow 0}{\lim } \mathcal{F}_{\varepsilon} (u) =
   2 \|u\|^p_{L^p}  \]
to hold for every $u$ in a prescribed subspace of $L^p (\R^N)$. In the
canonical smooth regime of compactly supported functions ($u \in
C_c^{\infty} ( \R^N )$), we show that convergence occurs when two optimal conditions are satisfied: (i) a mass escape condition, and (ii) a short-range attenuation effect, expressed by the vanishing as $\varepsilon \to 0$ of the kernels $p$-moments in any fixed neighborhood of the origin.
This general framework recovers the classical Maz’ya–Shaposhnikova theorem for fractional-type kernels and extends the convergence result to a much broader class of interaction profiles, which may be non-symmetric and non-homogeneous. Furthermore, using a density argument that preserves the moment assumptions, we prove that the same necessary and sufficient conditions remain valid in the integer-order Sobolev setting ($u \in W^{1, p} (\R^N)$). Finally, by adapting the method to fractional Sobolev spaces $W^{s, p} (\R^N)$ with $s \in (0, 1)$, we recover the Maz’ya–Shaposhnikova formula and extend it under analogous abstract conditions on the family of kernels $\{ \rho_{\varepsilon} \}_{\varepsilon > 0}$.
\end{abstract}


\maketitle

\section{Introduction}
Given a family of non-negative kernels $\{\rho_{\varepsilon}\}_{\varepsilon > 0}$, where $\rho_{\varepsilon} : \R^N \to [0,+\infty)$, we consider the associated \emph{nonlocal Dirichlet energies} by  
\[
\mathcal{F}_{\varepsilon}(u) := \int_{\R^N \times \R^N} \rho_{\varepsilon}(y - x) \, |u(x) - u(y)|^p \, \mathrm{d}x \, \mathrm{d}y, \quad u \in L^p(\R^N),
\]
for $1 \leqslant p < \infty$ and $N \in \mathbb{N}$.  
These functionals arise in a wide range of applications, including L\'evy-type stochastic processes, peridynamic models in mechanics, image processing, and the study of fractional Sobolev norms.

In their seminal work~\cite{BBM}, Bourgain, Brezis, and Mironescu showed that for the family of \emph{fractional kernels}
\begin{equation}
  \label{eq:fractional-kernel-BBM} 
  \rho_{\varepsilon}(z) = (1 - \varepsilon) |z|^{- (N + \varepsilon p)}, 
  \qquad \varepsilon \in (0, 1),
\end{equation}
there exists $K_{p,N}>0$ such that
\begin{equation}
\label{eq:BBMFirst}
  \lim_{\varepsilon \to 1^{-}} \mathcal{F}_{\varepsilon}(u) = K_{p,N} \|\nabla u\|_{L^p}^p,
\end{equation}
for every $u \in W^{1,p}(\R^N)$.  
One year later, Maz'ya and Shaposhnikova in ~\cite{MS}, completed the picture by proving that for the family of fractional kernels
\begin{equation}
  \label{eq:fractional-kernel} 
  \rho_{\varepsilon}(z) = \frac{\varepsilon p}{|\mathbb{S}^{N - 1}|} |z|^{- (N + \varepsilon p)}, 
  \qquad \varepsilon \in (0, 1),
\end{equation}
one has
\begin{equation}
  \label{eq:MSFirst}
  \lim_{\varepsilon \to 0} \mathcal{F}_{\varepsilon}(u) = 2 \|u\|_{L^p}^p,
\end{equation}
for every $u \in W^{s,p}(\R^N)$, with $s \in (0, 1)$.  
The limit identity~\eqref{eq:MSFirst} is now widely known as the \emph{Maz'ya--Shaposhnikova (MS) formula}.

Various extensions of the BBM and MS formulas in Euclidean spaces have been
developed, yielding a large and diverse literature. We do not attempt to provide an exhaustive survey; instead, we highlight contributions most closely related to {\cite{BBM,MS}}.

A \(\Gamma\)-convergence result that interprets Sobolev and BV norms as limits of nonlocal integral functionals was established in \cite{Ponce} (see also \cite{Cri}). Integral and relaxation characterizations of Sobolev and BV
spaces, together with lower semicontinuity results for families of nonlocal
functionals converging to Sobolev norms were developed in
{\cite{MR2832587,MR3132740}}. The asymptotic behavior of the fractional
$s$-perimeter in the regime $s \searrow 0$ is analyzed in {\cite{MR3007726}},
where conditions and examples are provided that clarify when the limit exists
and how it relates to volume-type quantities.
A sustained line of work has extended BBM/MS-type asymptotics across the full
scale of fractional Sobolev spaces $W^{s, p}$, carefully treating both the $s
\nearrow 1$ (BBM-type) and the $s \searrow 0$ (MS-type) regimes and
addressing endpoint and integrability subtleties; representative works in this
direction include
{\cite{MR3485124,MR3556344,MR3749763,MR4011115,MR4275122,MR4482042,MR4525722}}.
Fourier-analytic and function-space methods (notably in Triebel--Lizorkin and related endpoint spaces) offer alternative --- and in
some cases sharper --- descriptions of fractional norms and delicate endpoint
limits; see for instance {\cite{MR4275122,MR4482042,MR4525722}}.
Magnetic
variants of the BBM/MS formulas have also been developed: several results
demonstrate that magnetically weighted Gagliardo-type seminorms converge to
the appropriate local magnetic energies in the corresponding limits
{\cite{MR3601583,Ngu,MR3975602}}. Within Micromagnetics, a BBM-type formula incorporating nonlocal antisymmetric exchange interactions (Dzyaloshinskii–Moriya interaction, DMI) has been established in \cite{DDFG24} (see also \cite{DFGL25} for further results concerning physically relevant kernels).
Other generalizations focus on geometric and growth flexibility. BBM/MS-type
formulas have been proved in sub-Riemannian and Carnot-type settings and in various metric-measure frameworks by exploiting homogeneous structures,
tangent analysis, or volume-growth and Bishop--Gromov--type controls;
Orlicz--Sobolev extensions that accommodate non-power growth have also been
investigated {\cite{MR4449863,MR4453966,MR4685023,MSPin}}. Further directions
encompass vector- and matrix-valued nonlocal functionals with applications to
elasticity and homogenization {\cite{MR3886626,MR3986928}}, refined asymptotic
expansions for special kernel classes such as radial kernels
{\cite{Fog,Fog1}}, and a BBM-type representation for functions of bounded
deformation (BD) that links nonlocal energies to symmetric gradients
{\cite{MR4507709}}.

Over the past three decades, significant progress has been made in first-order analysis on metric measure spaces, including the development of first-order Sobolev spaces, functions of bounded variation, and their connections with variational problems and partial differential equations; see, for instance, \cite{AT-T} and the references therein. In \cite[Remark 6]{MR1942116}, Brezis raised a question concerning the relationship between the BBM formula and Sobolev spaces in the setting of metric measure spaces. This issue was subsequently addressed in \cite{DMS19}, where new characterizations of Sobolev and BV spaces in PI spaces were provided, inspired by the BBM framework; see also \cite{LPZ1,MR4788002}. Related results had previously been obtained in Ahlfors-regular spaces \cite{MR3299669}, and further investigations were carried out in certain PI spaces locally resembling Euclidean spaces \cite{MR4375837}. More recently, \cite{MR4782147} showed that the essential assumptions for obtaining \eqref{eq:MSFirst} and \eqref{eq:BBMFirst} are Rademacher’s theorem and volume growth at infinity. Finally, \cite{MSPin} established a surprising link between the MS formula and the generalized Bishop–Gromov inequality in the framework of metric measure spaces.

Despite the central role of the power-law kernels \eqref{eq:fractional-kernel-BBM} and \eqref{eq:fractional-kernel}, many contemporary applications demand \emph{greater flexibility} in the choice of interaction kernels --- for instance, compactly supported or anisotropic kernels, or kernels exhibiting singular behavior different from the classical power-law regime.
Motivated by these needs, several works have investigated broader kernel classes and identified conditions under which BBM/MS-type limits continue to hold; see, in particular, the complete characterization of kernels for which \eqref{eq:BBMFirst} is valid in \cite{Davoli} (for $p=2$) and in \cite{Stefani} (for the remaining cases), as well as the radial-kernel analyses in \cite{Fog1,Fog}.

This naturally raises the following question.
{\medskip}

{\noindent}{\textbf{Kernel generality.}} \emph{For a given subspace $X^p (\R^N)  \subseteq L^p (\R^N)$, which families of
kernels $\{\rho_{\varepsilon} \}_{\varepsilon > 0}$ ensure that
\eqref{eq:MSFirst} holds for every $u \in X^p (\R^N) $}? \medskip

In this paper, we offer a systematic and sharp answer to that question. Our main contributions can be summarized as follows.
\begin{itemize}
  \item \textbf{Exact kernel criterion in the smooth setting
  (Theorem~\ref{thm:MS-smooth}).} We show that for arbitrary non-negative,
  measurable kernels $\rho_{\varepsilon}$, the convergence
  $\mathcal{F}_{\varepsilon} (u) \to 2 \hspace{0.17em} \|u\|_{L^p}^p$ for all
  $u \in C^{\infty}_c (\R^N)$ is governed by two simple moment-type conditions on the kernels: a mass-escape condition and a short-range attenuation effect, expressed through the vanishing of suitably rescaled $p$-moments near the origin (see \eqref{eq:kernels-R_fix} or \eqref{eq:kernels-R_lim}). This formulation both unifies and extends the classical fractional-kernel case.\medskip
  
  \item \textbf{Extension to $W^{1, p} (\R^N)$
  (Corollary~\ref{cor:MS-sobolev}).} Remarkably, using a density argument that preserves the moment hypotheses, we show that the very same kernel conditions are necessary and sufficient for the convergence  $\mathcal{F}_{\varepsilon} (u) \to 2 \hspace{0.17em} \|u\|_{L^p}^p$ on the whole Sobolev space $W^{1, p} (\R^N)$. In other words, passing from compactly supported smooth functions to the Sobolev class $W^{1, p} (\R^N)$ does not require any stronger assumptions on the kernels.\medskip
  
  \item \textbf{Fractional Sobolev regime (Theorem~\ref{thm:MS-frac}).} By suitably adapting our moment assumptions to the weaker $W^{s,p}(\R^N)$ framework with $s\in(0,1)$, we recover the Maz’ya--Shaposhnikova formula for a broad class of kernels that includes the classical fractional family. This yields a genuine extension of the fractional MS result. Furthermore, Theorem~\ref{3-Thm2} shows that, under an additional hypothesis on the admissible kernels $\{\rho_{\varepsilon} \}_{\varepsilon > 0}$, the conditions of Theorem~\ref{thm:MS-frac} are not merely sufficient but also necessary for a Maz’ya–Shaposhnikova–type identity.\medskip
\end{itemize}
Beyond these main results, we develop several tools of independent interest and a detailed discussion of the limitations arising in the fractional case, which highlights an open problem about the necessity of our kernel assumptions (see Remark~\ref{rmk:genuinegen}).

\subsection*{Applications} Although our study is purely analytic, the two kernel properties we single out are tailored to capture several mechanisms relevant in applications. In peridynamics and nonlocal elasticity, one often encounters compactly supported finite-horizon kernels whose effective interaction length may vary with a parameter; such kernels are covered by the mass–escape condition when their support drifts to infinity (see, e.g., \cite{Du2012,Nezza2012,Silling2000}. In probability and materials models, long-tailed (Lévy-type) kernels are prototypical examples of mass concentration at infinity and satisfy our hypotheses as in the classical fractional case~\cite{Applebaum2009}. Finally, in image processing, anisotropic or non-homogeneous kernels are vital for tasks like image denoising \cite{Buades, Gilboa2009, Maleki2013}.

\subsection*{Structure of the paper.} Section~\ref{sec:2} states our principal results and places them in the context of the existing literature. There, we treat convergence both in the smooth setting and in integer-order Sobolev spaces, proving Theorem~\ref{thm:MS-smooth} and Corollary~\ref{cor:MS-sobolev}. Section~\ref{sec:FS} then addresses the fractional Sobolev case: in Theorem~\ref{thm:MS-frac} we provide a sufficient condition for the validity of the Maz'ya--Shaposhnikova formula, while in Theorem~\ref{3-Thm2} we prove that, under an additional assumption on the admissible kernels $\{\rho_{\varepsilon}\}_{\varepsilon>0}$, these conditions are not only sufficient but also necessary for a Maz'ya--Shaposhnikova-type formula to hold.

\section{The MS formula in the smooth and integer-order Sobolev
settings}\label{sec:2}
Let $p \in [1,  \infty)$ and $N \in \N$. Consider a family of nonnegative measurable kernels $\{\rho_{\varepsilon} : \R^N \to [0,
\infty)\}_{\varepsilon > 0}$. For every $\varepsilon > 0$ and every measurable
function $u$, we define the associated nonlocal energy functional
\begin{equation}
  \label{energy} \mathcal{F}_{\varepsilon} (u):= \int_{\R^N
  \times \R^N} \rho_{\varepsilon}  (y - x) \, |u (x) - u (y) |^p
  \d x \hspace{0.17em} \d y,
\end{equation}
with the convention that the above energy takes infinite values whenever $u$ does not belong to its domain:
\begin{equation}\label{eq:NL-functional}
\mathcal{X}^p_{\eps} (\R^N):= \{u \in L^p (\R^N) \,:\,
  \mathcal{F}_{\varepsilon} (u) < + \infty\}.
\end{equation}
Given $\varepsilon_0>0$, we further introduce the function space:
\begin{equation}
  \label{f-space} 
  X^p_{\eps_0} (\R^N)
 := \bigcap_{\eps \leqslant \eps_0} \mathcal{X}^p_{\eps} (\R^N) .
\end{equation}
Observe that, without further assumptions on the kernels $\rho_{\varepsilon}$,
the spaces $\mathcal{X}_{\varepsilon}^p (\R^N)$ need not be nested as $\varepsilon$
varies, whereas by construction the intersections $X_{\varepsilon_0}^p (\R^N)$
form a decreasing family in $\varepsilon_0$.

Depending on the regularity of the function $u$, we establish a generalized
version of the Maz'ya--Shaposhnikova formula {\cite{MS}}.

\subsection{The MS formula in the smooth settings}\label{subsec2.1}Our first
main result concerns functions $u \in C^{\infty}_c (\R^N)$.

\begin{theorem}
  \label{thm:MS-smooth}The following three conditions are equivalent:
  \begin{enumerate}
    \item \label{cond:R-fix}{\textbf{{\emph{Uniform moment conditions.}}}}
    {\noindent}For every fixed radius $R > 0$, as $\varepsilon \to 0$
    the kernels $\{ \rho_{\varepsilon} \}_{\varepsilon > 0}$ satisfy
    \begin{equation}
      \label{eq:kernels-R_fix} \lim_{\varepsilon \to 0}  \int_{|z| > R}
      \rho_{\varepsilon} (z) \hspace{0.17em} \d z = 1 \quad \text{and}
      \quad \lim_{\varepsilon \to 0}  \int_{|z| < R} |z|^p \rho_{\varepsilon}
      (z) \hspace{0.17em} \d z = 0.
    \end{equation}
    \item \label{cond:R-lim}{\textbf{{\emph{Iterated limits conditions.}}}}
    {\noindent}The double limits in the order ``first $\varepsilon \rightarrow
    0$, then $R \rightarrow \infty$'' satisfy
    \begin{equation}
      \label{eq:kernels-R_lim} \lim_{R \to \infty} \lim_{\varepsilon \to 0} 
      \int_{|z| > R} \rho_{\varepsilon} (z) \hspace{0.17em} \d z = 1
      \quad \text{and} \quad \lim_{R \to \infty} \lim_{\varepsilon \to 0} 
      \int_{|z| < R} |z|^p \rho_{\varepsilon} (z) \hspace{0.17em} \d z
      = 0.
    \end{equation}
    \item
    \label{eq:MS+II}\label{cond:R-lim3}{\textbf{{\emph{Maz'ya--Shaposhnikova
    formula.}}}} {\noindent}For every smooth, compactly supported function $u
    \in C^{\infty}_c (\R^N)$, $\mathcal{F}_{\varepsilon} (u)$ is well-defined
    for $\varepsilon > 0$ sufficiently small, and
    \begin{equation}
      \label{eq:MS-smooth} \lim_{\varepsilon \to 0} \mathcal{F}_{\varepsilon}
      (u) = 2 \|u\|^p_{L^p (\R^N)} .
    \end{equation}
    Moreover, for each fixed $R > 0$, there holds
    \begin{equation}\label{condex}
      \lim_{\varepsilon \to 0}  \int_{|z| < R} \int_{\R^N}
      \rho_{\varepsilon} (z)  |u (x + z) - u (x) |^p \hspace{0.17em}
      \d x \hspace{0.17em} \d z = 0.
    \end{equation}
  \end{enumerate}
\end{theorem}
  Theorem~\ref{thm:MS-smooth} both characterizes and extends the Maz'ya–Shaposhnikova framework: it gives a precise description of those kernels whose associated  {\emph{nonlocal}} Dirichlet energies converge to the {\emph{local}} $L^p$-norm in the regime where only long-range interactions remain relevant.
  The proof is presented in Section~\ref{proof:thmMssmooth}; before that we collect a few remarks and illustrative examples.

\begin{remark}[Finite L\'evy measures]
\label{lemma:inclusion-smooth}
We preliminarily observe that, for sufficiently small $\varepsilon$, each
kernel $\rho_{\varepsilon}$ fulfilling the fixed--radius moment conditions \eqref{eq:kernels-R_fix}  is a finite Lévy measure, i.e.
  \[ \int_{\R^N} \min \{ 1, | z |^p \}  \rho_{\varepsilon}
     (z) \, \d z \, < \, + \infty.\]
     We refer the reader to \cite{Fog} for a complete overview of the subject.
In particular, the first claim in
item~(\ref{eq:MS+II}) is automatically satisfied, for instance, by \cite[Proposition 3.12]{Fog}. In fact, by  conditions \eqref{eq:kernels-R_fix}, there exist $\varepsilon_0 > 0$ and $c_{\rho} > 0$, such that 
\[ \int_{\R^N} \min \{ 1, | z |^p \} \, \rho_{\varepsilon} 
(z) \, \d z < c_{\rho} \quad \text{for~all} \quad 0 < \varepsilon 
\leqslant \varepsilon_0 . \] 
Let $u \in W^{1, p} (\R^N)$ and $z\in\mathbb{R}^N$. We denote by $\tau_z u (x) := u (x 
+ z)$ the translation of $u$ by $z$. Standard Sobolev estimates yield, for every $z \in 
\R^N$, 
\[ \| \tau_z u - u \|_{L^p}^p \, \leqslant \, 2^p \| u 
\|_{L^p}^p, \qquad \| \tau_z u - u \|_{L^p}^p \leqslant | z |^p \| \nabla 
u \|_{L^p}^p. \] 
Hence 
\[ \| \tau_z u - u \|_{L^p}^p \leqslant 2^p \, \min \{ 1, | z 
|^p \} \, \| u \|_{W^{1, p}}^p . \] 
Integrating this estimate against $\rho_{\varepsilon} (z)$ entails 
\[ \mathcal{F}_{\varepsilon} (u) \leqslant 2^p\| u \|_{W^{1, p}}^p \int_{\R^N} 
\min \{1, |z|^p \} \rho_{\varepsilon} (z) \hspace{0.17em} \d z < 
c_{\rho} \| u \|_{W^{1, p}}^p . \] 
Thus $u \in \mathcal{X}_{\varepsilon}^p (\R^N)$ for every $\varepsilon 
\leqslant \varepsilon_0$, and the inclusion $W^{1, p} (\R^N) 
\subset X_{\varepsilon_0}^p (\R^N)$ follows. 
\end{remark}

\begin{remark}[On the uniform moment conditions]
We refer to the two conditions in \eqref{eq:kernels-R_fix}, which ensure the validity of the generalized Maz'ya--Shaposhnikova formula in Theorem \ref{thm:MS-smooth}, as the \textit{mass-escape condition} and the \textit{short-range attenuation effect}. 
The contribution of the nonlocal energy \eqref{eq:NL-functional} that yields the MS formula comes from regions far away from the origin. For this reason, in the case of functions with support on the whole space $\R^N$, it is essential to work with kernels that, in the limit, concentrate their mass at infinity. This motivates the first condition in \eqref{eq:kernels-R_fix}, referred to as \textit{mass-escape condition}.  

On the other hand, the second condition in \eqref{eq:kernels-R_fix}, instead, specifies how singular the kernel can be near the origin. It also requires that the short-range contribution vanish in the limit, ensuring that the total mass is concentrated away from the origin—hence the term \textit{short-range attenuation}. Note, however, that the kernel is not necessarily singular at the origin a priori. For instance, the kernel
\begin{equation*}
        \rho_{\varepsilon}(z) :=  
        \begin{cases} 
            \varepsilon  & \text{if } z \in B_1(0), \\[6pt] 
            \dfrac{\varepsilon p}{\mid \mathbb{S}^{N - 1} \mid} \, |z|^{-(N + \varepsilon p)} & \text{if } z \in \R^N \setminus B_1(0). 
        \end{cases} 
\end{equation*}
satisfies the conditions \eqref{eq:kernels-R_fix} and therefore guarantees the validity of the MS formula. Further examples on kernels are provided in the Remark below.
\end{remark}

\begin{remark}[A few examples of kernels satisfying \eqref{eq:kernels-R_fix}]\label{Rmk:examples}
We collect below some families of kernels complying with \eqref{eq:kernels-R_fix} (equivalently, with \eqref{cond:R-lim}).
\begin{enumerate}
  \item Canonical examples are the {\emph{fractional kernels}} $\rho_{\varepsilon}
  (z) := \frac{\varepsilon p}{\mid \mathbb{S}^{N - 1} \mid} \, |
  z |^{- (N + \varepsilon p)}$, which indeed satisfy for each fixed $R > 0$
  \[ \int_{\mid z \mid > R} \rho_{\varepsilon} (z) \, \d z = R^{-
     \varepsilon p} \, \xrightarrow{\, \varepsilon
     \to 0 \,} 1 \]
  and
  \[ \int_{\mid z \mid < R} | z |^p \, \rho_{\varepsilon} (z) \,
     \d z = \eps p \frac{R^{p -
     \varepsilon p}}{p - \varepsilon p} \,
     \xrightarrow{\, \varepsilon \to 0 \,} 0. \]
  \item  Another interesting family of examples can be constructed as follows. For any $\phi\in C^{\infty}_c(\R^N)$, such that $\phi\geqslant 0$, and $\mathrm{supp}(\phi)\subset B(0,1)$, with $\int_{\R^N}\phi=1$, we define $\rho_{\varepsilon}(z):=\phi(z-a_{\varepsilon}e_1)$, where $a_{\varepsilon}\to\infty$ as $\varepsilon\to 0$ and $e_1$ is the first element of the canonical basis of $\R^N$. It is straightforward to see that this family of kernels satisfies \eqref{eq:kernels-R_fix}. 
 
 It is interesting to point out that the compactness of the support of $\phi$ can be removed in certain cases, for instance if we choose $\phi(x)=(2\pi)^{-N/2}e^{-{|x|^2}/{2}}$ and $a_\varepsilon=1/\varepsilon$. In fact, more generally, the two conditions in  \eqref{eq:kernels-R_fix} are satisfied as soon as $\| \phi \|_{L^1(\R^N)}=1$. To see this, we argue as follows.\\
  We start proving
\[
\lim_{\varepsilon\to 0}\int_{|z|<R}|z|^p\,\rho_\varepsilon(z)\,\d z=0.
\]
Clearly,
\[
0 \leqslant \int_{|z|<R}|z|^p\,\rho_\varepsilon(z)\,\d z
\leqslant R^p \int_{|z|<R}\rho_\varepsilon(z)\,\d z.
\]
Thus, it suffices to show that
\[
\lim_{\varepsilon\to 0}\int_{|z|<R}\rho_\varepsilon(z)\,\d z = 0.
\]
To prove this, set \(a_\varepsilon:=1/{\varepsilon}\) and make the change of variables $y=z-a_\varepsilon e_1$. We obtain
\[
\int_{|z|<R}\rho_\varepsilon(z)\,\d z
=\int_{|y+a_\varepsilon e_1 |<R}\phi(y)\,\d y.
\]
For each fixed \(y\in\mathbb{R}^N\) we have \(|y+a_\varepsilon e_1 |\to\infty\) as \(\varepsilon\to0\), so the indicator \(\mathbf{1}_{\{|y+a_\varepsilon e_1|<R\}}\) converges pointwise to \(0\). Because \(0\leqslant \mathbf{1}_{\{|y+a_\varepsilon e_1|<R\}}\phi(y)\leqslant \phi(y)\) and \(\phi\in L^1(\mathbb{R}^N)\) with \(\int_{\mathbb{R}^N}\phi=1\), we apply the Dominated Convergence Theorem to get
\[
\lim_{\varepsilon\to0}\int_{|y+a_\varepsilon e_1 |<R}\phi(y)\,\d y=0.
\]
Combining this with the previous inequalities, we infer the conclusion.\\ 
Let us now prove
\[
\lim_{\varepsilon\to 0}\int_{|z|>R}\rho_\varepsilon(z)\,\d z=1.
\]
As before, with the change of variables \(y=z-a_\varepsilon e_1\), we get
\[
\int_{|z|>R}\rho_\varepsilon(z)\,\d z
= \int_{|y+a_\varepsilon e_1|>R}\phi(y)\,\d y.
\]
For each fixed \(y\) the indicator \(\mathbf{1}_{\{|y+a_\varepsilon e_1|>R\}}\) converges pointwise to \(1\) as \(\varepsilon\to0\). Since $0\leqslant \mathbf{1}_{\{|y+a_\varepsilon e_1|>R\}}\phi(y)\leqslant \phi(y)$
and \(\phi\) is integrable, the Dominated Convergence Theorem yields
\[
\lim_{\varepsilon\to0}\int_{|y+a_\varepsilon e_1|>R}\phi(y)\,\d y
=\int_{\mathbb{R}^N}\phi(y)\,\d y=1,
\]
which proves the claim.
  
   \item An alternative example with a slightly stronger singularity in the origin is given by:
  \begin{equation*}
        \rho_{\varepsilon}
  (z) :=  \begin{cases} \frac{\varepsilon^2 p^2}{\mid \mathbb{S}^{N - 1} \mid} \, \log(1/|z|) | 
  z |^{- (N + \varepsilon p)}  & \textrm{in } B_1(0), \\ \frac{\varepsilon p}{\mid \mathbb{S}^{N - 1} \mid} \, |
  z |^{- (N + \varepsilon p)}&  \textrm{in } \R^N \setminus B_1(0). 
\end{cases} 
    \end{equation*}
For this example, it is easier to show that conditions \eqref{eq:kernels-R_lim} are satisfied. Indeed, in establishing \eqref{eq:kernels-R_lim} we can always assume that $R>1$. 
We argue as follows. If $R> 1$ then
\begin{align*}
\int_{|z|>R}\rho_{\varepsilon}(z)\ \d z=\varepsilon p\int_{R}^{\infty} r^{-\varepsilon p-1}\ \d r = \frac{1}{R^{\varepsilon p}}\to 1\quad \mbox{as}\ \varepsilon \to 0.
\end{align*}
This shows that the first condition in  \eqref{eq:kernels-R_lim} is satisfied.

We now prove that the second condition in \eqref{eq:kernels-R_lim} is fulfilled. Again, we can assume $R>1$. We obtain
\[
\int_{|z|<R}|z|^p\rho_{\varepsilon}(z)\ \d z= A_{\varepsilon} +B_{\varepsilon}:= \int_{1\leqslant |z|<R}|z|^p\rho_{\varepsilon}(z)\, \d z+\int_{|z|<1}|z|^p\rho_{\varepsilon}(z) \d z .
\]
Passing to polar coordinates, we find
\begin{equation*}
A_{\varepsilon}=\varepsilon p\int_1^R r^{p-\varepsilon p-1}\ \d r=\frac{\varepsilon p}{p-\varepsilon p}(R^{p-\varepsilon p}-1)\to 0\ \mbox{as}\ \varepsilon\to 0
\end{equation*}
\begin{align*}
&B_{\varepsilon}=(\varepsilon p)^2 \int_0^{1} \log(1/r)\, r^{p-1-\varepsilon p}\, \mathrm \d r=\frac{\varepsilon^2}{(1-\varepsilon)^2}\to 0\ \mbox{as}\ \varepsilon\to 0 ,
\end{align*}
where for the computation of $B_\varepsilon$ we simply used that for $\alpha>0$ a primitive of $r^\alpha \log(r)$ is given by the function $\frac{r^{\alpha + 1}}{\alpha + 1} \left( \log  r - \frac{1}{\alpha + 1}
\right)$.
\end{enumerate}
\end{remark}

\subsection{Proof of Theorem~\ref{thm:MS-smooth}} \label{proof:thmMssmooth}
\noindent In the following, for every $R > 0$, we will often use the decomposition
\begin{align}
  \mathcal{F}_{\varepsilon} (u) & =  \int_{|z| > R} \rho_{\varepsilon} (z) 
  \int_{x \in \R^N} |u (x + z) - u (x) |^p  \d x
  \hspace{0.17em} \d z \nonumber\\
  &  \qquad \qquad \qquad + \int_{|z| < R} \rho_{\varepsilon} (z)  \int_{x
  \in \R^N} |u (x + z) - u (x) |^p \d x \hspace{0.17em}
  \d z \nonumber\\
  & =:  I_{\varepsilon, R} [u] + \II_{\varepsilon, R} [u]. 
  \label{eq:decomp}
\end{align}
The proof of Theorem~\ref{thm:MS-smooth} rests on the following auxiliary result.{\smallskip} 
\begin{lemma}
  \label{lemma:(ii)-smooth}Let $R > 0$. The following statements are
  equivalent:
  \begin{enumerate}
    \item \label{i}
    \begin{equation}
      \lim_{\varepsilon \to 0}  \int_{|z| < R} |z|^p
      \rho_{\varepsilon} (z) \hspace{0.17em} \d z = 0.
    \end{equation}
    \item \label{ii} For every $u \in C^{\infty}_c (\R^N)$
    \begin{equation}
      \lim_{\varepsilon \to 0} \II_{\varepsilon, R} [u] =
      \lim_{\varepsilon \to 0}  \int_{|z| < R} \int_{\R^N}
      \rho_{\varepsilon} (z)  |u (x + z) - u (x) |^p \d x
      \hspace{0.17em} \d z = 0.
    \end{equation}
  \end{enumerate}
\end{lemma}

\begin{proof}
  {\textbf{(\ref{i} $\Rightarrow$ \ref{ii}).}} For any $u \in C_c^{\infty}
  (\R^N)$, classical Sobolev inequalities give
  \[ \int_{\R^N} |u (x + z) - u (x) |^p  \d x \leqslant |z|^p  \|
     \nabla u\|^p_{L^p (\R^N)} . \]
  Therefore
  \[ 0\leqslant \II_{\varepsilon, R} [u] \leqslant \| \nabla u\|^p_{L^p ( \R^N
     )}  \int_{|z| < R} |z|^p \rho_{\varepsilon} (z) \hspace{0.17em}
     \d z, \]
  and the right‑hand side vanishes as $\varepsilon \to 0$ by
  assumption.{\smallskip}
  
  {\noindent}{\textbf{(\ref{ii} $\Rightarrow$ \ref{i}).}} We begin by
  decomposing the integral over $\{ 0 < \mid z \mid < R \}$ into
   $\{ 0 < \mid z \mid < \delta \}$ and $\{ \delta < \mid z \mid < R
  \}$, with $\delta \in (0, R)$ to be chosen later. We handle these two regions
  separately.\\
  We first estimate $\II_{\varepsilon}$ in the {\textbf{small
  range}} $0 < | z | < \delta$. Let $u_1 \in C_c^{\infty} ( \R^N
  )$. A first-order Taylor expansion of $\nabla u_1$ at $x
  \in \R^N$ yields, for each $| z | < \delta$,
  \[ \left| \int_0^1 \nabla u_1 (x
    + tz) \cdot \frac{z}{|z|} \hspace{0.17em} \d t \right|^p \,
     \geqslant \frac{1}{2^{p - 1}}  \left| \nabla u_1 (x) \cdot \frac{z}{|z|}
     \right|^p \, - \, \delta^p \, \| \nabla^2 u_1
     \|_{L^{\infty}}^p . \]
  It follows that
  \begin{align}
    \II_{\varepsilon, \delta} [u_1] & :=  \int_{|z| < \delta} |z|^p
    \rho_{\varepsilon} (z)  \int_{\R^N} \left| \int_0^1 \nabla u_1 (x
    + tz) \cdot \frac{z}{|z|} \hspace{0.17em} \d t \right|^p
    \d x \hspace{0.17em} \d z \nonumber\\
    & \geqslant  \int_{|z| < \delta} |z|^p \rho_{\varepsilon} (z) 
    \int_{B_{1 + \delta}} \left( \frac{1}{2^{p - 1}}  \left| \nabla u_1 (x)
    \cdot \frac{z}{|z|} \right|^p - \delta^p  \| \nabla^2 u_1
    \|_{L^{\infty}}^p \right) \d x \hspace{0.17em} \d z .
    \nonumber
  \end{align}
  Next, we observe that if we choose $u_1$ to be a radial bump, i.e., $u_1 (x)
  = g (|x|)$ for a suitable function $g \in C_c^{\infty} (\R_+)$ such that
  $\mathrm{supp} \,
  g \subseteq [0, 1)$, then, $\nabla u_1 (x) = g' (|x|)
  \frac{x}{|x|}$  for every $x \neq 0$ and with a change of variables, we find
  that
  \begin{equation}
    \II_{\varepsilon, \delta} [u_1] \geqslant \alpha_p (u_1)  \int_{|z| <
    \delta} |z|^p \rho_{\varepsilon} (z) \hspace{0.17em} \d z
    \label{eq:bound-delta}
  \end{equation}
  with
  \begin{equation}
    \alpha_p (u_1):= \left( \frac{1}{2^{p - 1}}  \int_{B_{1 + \delta}}
    \left| g' (|x|) \frac{x_1}{|x|} \right|^p \d x \right) - | B_{1 +
    \delta} | \delta^p  \| \nabla^2 u_1 \|_{L^{\infty}}^p .
  \end{equation}
  In particular, if we choose the radial profile $g (s) = \frac{1}{2} \, s^2$
  for $s \leqslant \frac{1}{2}$, then in $B_{1 / 2}$ we have $g' (|x|)
  \frac{x_1}{|x|} = x_1$, and therefore one obtains the lower bound
  \[ \alpha_p (u_1) \geqslant \frac{1}{2^{p - 1}}  \int_{B_{1 / 2}} | x_1 |^p
     \d x - | B_{1 + \delta} | \delta^p  \| \nabla^2 u_1
     \|_{L^{\infty}}^p, \]
  which is strictly positive provided $\delta$ is chosen sufficiently small.
  Combining this positivity with assumption~\ref{ii} and
  \eqref{eq:bound-delta}, we conclude that
  \begin{equation}
    \label{eq:cond-delta} \lim_{\varepsilon \to 0}  \int_{|z| <
    \delta} |z|^p \rho_{\varepsilon} (z) \hspace{0.17em} \d z = 0.
  \end{equation}
  We now turn to estimating $\II_{\varepsilon}$ in the
  {\textbf{intermediate range}} $\delta < | z | < R$. We choose $u_2 \in
  C^{\infty}_c (\R^N)$ such that $\supp u_2 \subseteq B_{\delta / 2} (0)$. We
  find that
  \begin{align}
    \II_{\varepsilon, R} [u_2] & \geqslant  \int_{\delta < |z| < R}
    \int_{|x| > \delta / 2} \rho_{\varepsilon} (z)  | u_2 (x + z) - u_2 (x)
    |^p \d x \hspace{0.17em} \d z \nonumber\\
    & = \int_{\delta < |z| < R} \int_{|x| > \delta / 2} \rho_{\varepsilon}
    (z)  | u_2 (x + z) |^p \d x \hspace{0.17em} \d z
    \nonumber\\
    & \geqslant \frac{1}{R^p}  \int_{\delta < |z| < R} |z|^p
    \rho_{\varepsilon} (z)  \int_{|x - z| > \delta / 2} | u_2 (x) |^p
    \d x \hspace{0.17em} \d z \nonumber\\
    & =  \frac{1}{R^p}  \int_{\delta < |z| < R} |z|^p \rho_{\varepsilon} (z)
    \int_{(\R^N \setminus B_{\delta / 2} (z)) \cap B_{\delta / 2}
    (0)} | u_2 (x) |^p \d x \hspace{0.17em} \d z. \nonumber
  \end{align}
  Observe that for $|z| > \delta$, we have $(\R^N \setminus
  B_{\delta / 2} (z)) \cap B_{\delta / 2} (0) \equiv B_{\delta / 2} (0)$.
  Therefore
  \begin{align}
    \II_{\varepsilon, R} [u_2] & \geqslant  \frac{1}{R^p}  \int_{\delta <
    |z| < R} |z|^p \rho_{\varepsilon} (z)  \int_{B_{\delta / 2} (0)} | u_2 (x)
    |^p \d x \hspace{0.17em} \d z \nonumber\\
    & = \frac{\| u_2 \|_{L^p}^p}{R^p}  \int_{\delta < |z| < R} |z|^p
    \rho_{\varepsilon} (z) \hspace{0.17em} \d z. \nonumber
  \end{align}
  Under assumption~\ref{ii}, taking the limit on both sides as $\varepsilon
  \to 0$ forces
  \begin{equation}
    \label{eq:cond-fuoridelta} \lim_{\varepsilon \to 0}  \int_{\delta
    < |z| < R} |z|^p \rho_{\varepsilon} (z) \hspace{0.17em} \d z = 0.
  \end{equation}
  The implication (\ref{ii} $\Rightarrow$ \ref{i}) now follows by combining
  \eqref{eq:cond-delta} and \eqref{eq:cond-fuoridelta}.
\end{proof}

\begin{proof}[Proof of Theorem~\ref{thm:MS-smooth}.]
  {\noindent}{\textbf{(\ref{cond:R-fix} $\Rightarrow$ \ref{cond:R-lim}).}}
  Trivial by taking the limits in the prescribed order.{\smallskip}
  
  {\noindent}{\textbf{(\ref{cond:R-lim} $\Rightarrow$ \ref{cond:R-lim3}).}}
  First, by monotonicity in $R > 0$, the second condition in equation
  \eqref{eq:kernels-R_lim} is equivalent to the requirement
  \begin{equation}
    \lim_{\varepsilon \to 0}  \int_{|z| < R} |z|^p \rho_{\varepsilon} (z)
    \hspace{0.17em} \d z = 0 \quad \text{for every } R > 0.
  \end{equation}
  Let $u \in C^{\infty}_c (\R^N)$ and decompose the nonlocal energy as in
  \eqref{eq:decomp}, i.e., $\mathcal{F}_{\veps} (u) = I_{\veps, R} [u] + I 
  I_{\veps, R} [u]$. By Lemma~\ref{lemma:(ii)-smooth}, for each fixed $R > 0$, we have that
  \begin{equation}
    \lim_{\varepsilon \to 0} \II_{\veps, R} [u] = \lim_{\varepsilon
    \to 0}  \int_{|z| < R} \int_{\R^N} \rho_{\varepsilon} (z) 
    |u (x + z) - u (x) |^p \hspace{0.17em} \d x \hspace{0.17em}
    \d z = 0. \label{eq:II-smooth}
  \end{equation}
  Next, choose $R$ so large that $\supp u \subset B_{R / 2}$. Then,
  \begin{align}
    I_{\varepsilon, R} [u] & =  \int_{|z| > R} \rho_{\varepsilon} (z) \left(
    \int_{|x| < R / 2} | u (x + z) - u (x) |^p  \d x + \int_{|x| > R / 2}  |
    u (x + z) |^p  \d x \right)  \d z \nonumber\\
    & =  \int_{|z| > R} \rho_{\varepsilon} (z) \left( \int_{|x| < R / 2} | u
    (x) |^p  \d x + \int_{|x| > R / 2} | u (x + z) |^p  \d x \right)  \d
    z \nonumber\\
    & = 2 \hspace{0.17em} \|u\|_{L^p}^p  \int_{|z| > R} \rho_{\varepsilon}
    (z)  \hspace{0.17em} \d z,  \label{eq:I-smooth}
  \end{align}
  where we used that the shifted support $x + z$ remains outside $B_{R / 2}$
  when $|z| > R$ and $|x| < R / 2$, i.e., $|x + z| \geqslant |z| - |x|
  \geqslant R / 2$.
  
  Combining \eqref{eq:II-smooth} and \eqref{eq:I-smooth}, for every $R > 0$
  such that $\supp u \subset\subset B_{R / 2}$, we get that
  \[ \lim_{\varepsilon \to 0} \mathcal{F}_{\varepsilon} (u) = 2
     \|u\|_{L^p}^p \lim_{\varepsilon \to 0}  \int_{|z| > R}
     \rho_{\varepsilon} (z) \hspace{0.17em} \d z. \]
  Finally, letting $R \rightarrow \infty$ and using the first condition in
  item~\ref{cond:R-lim}, yields that
  \[ \lim_{\varepsilon \to 0} \mathcal{F}_{\varepsilon} (u) = 2
     \|u\|_{L^p}^p \lim_{R \rightarrow \infty} \lim_{\varepsilon \rightarrow
     0}  \int_{|z| > R} \rho_{\varepsilon} (z) \hspace{0.17em} \d z =
     2 \|u\|_{L^p}^p \]
  for every $u\in C^{\infty}_c(\R^N)$, as claimed. This concludes the proof of the implication (\ref{cond:R-lim}
  $\Rightarrow$ \ref{cond:R-lim3}).{\smallskip}
  
  {\noindent}{\textbf{(\ref{cond:R-lim3} $\Rightarrow$ \ref{cond:R-fix}).}}
  Let $u \in C_c^{\infty} (\R^N)$. By hypothesis, we have
  \begin{equation}
    \label{eq:hp-1} \lim_{\varepsilon \to 0} \mathcal{F}_{\varepsilon}
    (u) = 2 \|u\|_{L^p}^p
  \end{equation}
  and for every fixed $R > 0$,
  \begin{equation}
    \label{eq:hp-2} \lim_{\varepsilon \to 0} \II_{\varepsilon, R} [u]
    = 0.
  \end{equation}
  According to Lemma~\ref{lemma:(ii)-smooth}, the previous relation entails
  the second relation in item~\ref{cond:R-fix}, i.e., that
  \begin{equation}
    \lim_{\varepsilon \to 0}  \int_{|z| < R} |z|^p \rho_{\varepsilon}
    (z) \hspace{0.17em} \d z = 0 \quad \text{for every } R > 0.
  \end{equation}
  It remains to show the complementary tail condition
  \[ \lim_{\varepsilon \to 0}  \int_{|z| > R} \rho_{\varepsilon} (z)
     \hspace{0.17em} \d z = 1 \quad \text{for every } R > 0. \]
  To this end, fix $R > 0$ and choose $u \in C_c^{\infty} (\R^N)$
  with $\mathrm{supp} \, u \subset B_{R / 2}$. A direct
  computation---identical to that which produced equation \eqref{eq:I-smooth}---gives
  \[ \mathcal{F}_{\varepsilon} (u) = 2 \|u\|_{L^p}^p  \int_{|z| > R}
     \rho_{\varepsilon} (z) \hspace{0.17em} \d z + \II_{\varepsilon,
     R}[u]. \]
  Passing to the limit as $\varepsilon \to 0$ and invoking
  \eqref{eq:hp-1}-\eqref{eq:hp-2} shows
  \[ 2 \|u\|_{L^p}^p = 2 \|u\|_{L^p}^p \lim_{\varepsilon \to 0} 
     \int_{|z| > R} \rho_{\varepsilon} (z) \hspace{0.17em} \d z, \]
  from which the tail condition follows at once. This completes the proof of
  Theorem~\ref{thm:MS-smooth}.
\end{proof}

We conclude this section with three corollaries showing that, under further integrability assumptions on the kernels, condition \eqref{condex} can be dropped.
\begin{corollary}
    
\label{rmk:kernelsamearea}
    Under the same assumptions as in Theorem~\ref{thm:MS-smooth}, assume in addition that $\int_{\R^N} \rho_{\varepsilon} (x) 
\hspace{0.17em} \mathrm{d} x = 1$ for every $\varepsilon > 0$. Then, the condition \eqref{condex} in Theorem~\ref{thm:MS-smooth} is not needed, since it follows automatically from the convergence assumption. Precisely, 
under the uniform normalization condition
\[ \int_{\R^N} \rho_{\varepsilon} (x)  \hspace{0.17em} \mathrm{d} x =
   1, \quad \forall \varepsilon > 0, \]
the Maz'ya--Shaposhnikova (MS) limit $\lim _{\varepsilon \rightarrow 0}
\mathcal{F}_{\varepsilon} (u) = 2 \|u\|_{L^p}^p$ for all $u \in C_c^{\infty}
( \R^N )$ implies the fixed-radius moment conditions
\[ \lim_{\varepsilon \rightarrow 0} \int_{\mid z \mid > R} \rho_{\varepsilon}
   (z) \hspace{0.17em} \mathrm{d} z = 1 \quad \text{and} \quad \lim_{\varepsilon \rightarrow
   0} \int_{\mid z \mid \leqslant R} | z |^p \rho_{\varepsilon}
   (z)\hspace{0.17em} \mathrm{d} z = 0 \]
for all $R > 0$.
\end{corollary} 
\begin{proof}
It is convenient to introduce some notation. For a function $u \in C_c^{\infty}
( \R^N)$,
set
\[ A_u (z) : = \int_{\R^N} | u (x + z) - u (x) |^p \hspace{0.17em} \mathrm{d} x, \quad g_u (z)
   : = 2 \|u\|_{L^p}^p - A_u (z) . \]
Observe that $A_u$ is continuous in $z$, $0 \leqslant A_u (z) \leqslant 2^{p}
\|u\|_{L^p}^p$, and if $\ensuremath{\operatorname{supp}}  \; u \subset B (0,
R_0)$ for some $R_0 > 0$, then $A_u (z) = 2 \|u\|_{L^p}^p$ for $| z | > 2 R_0$;
hence $g_u$ is continuous, non-negative, and compactly supported.{\smallskip}

{\noindent}\text{{\bfseries{Step 1: Towards the mass escape condition
\eqref{eq:kernels-R_lim}.}}} We show that for every fixed $R > 0$,
\[ \lim_{\varepsilon \to 0}  \int_{|z| \leqslant R} \rho_{\varepsilon} (z)
   \hspace{0.17em} \mathrm{d} z = 0, \qquad \text{or equivalently} \qquad \lim_{\varepsilon \to
   0}  \int_{|z| > R} \rho_{\varepsilon} (z) \hspace{0.17em} \mathrm{d} z = 1. \]
From the MS hypothesis and the normalization $\int_{\R^N}
\rho_{\varepsilon} = 1$ we have, for every fixed $u \in C_c^{\infty} (\R^N)$,
\begin{equation}
  \lim_{\varepsilon \to 0}  \int_{\R^N} \rho_{\varepsilon} (z) A_u (z)
  \hspace{0.17em} \mathrm{d} z = 2 \|u\|_{L^p}^p \quad \Longleftrightarrow \quad
  \lim_{\varepsilon \to 0}  \int_{\R^N} \rho_{\varepsilon} (z) g_u (z)
  \hspace{0.17em} \mathrm{d} z = 0. \label{eq:JMS2}
\end{equation}
Let $R > 0$ be fixed. Choose numbers $0 < \rho < r$ such that $r > R + \rho$,
and set $s := r - \rho$. Note that $s > R$ by construction. Let $u \in
C_c^{\infty} (\R^N)$ be a radial function satisfying
\begin{equation}
  \mathrm{supp}\, u \subset \overline{B (0, r)}, \qquad u \geqslant 0, \qquad u
  > 0 \text{ on } B (0, r). \label{eq:testuprops}
\end{equation}
As a first step, we want to prove that $g_u$ is strictly positive on the
compact set $\overline{B (0, s)}$, so that $g_u$ attains a strictly positive
minimum in there:
\begin{equation}
  m_u (s) := \min_{|z| \leqslant s} g_u (z) > 0. \label{eq:minu}
\end{equation}
For that, we analyze the term
\[ A_u (z) = \int_{\R^N} |u (x + z) - u (x) |^p  \hspace{0.17em}
   \mathrm{d} x. \]
By the elementary lemma (if $a, b > 0$ and $p \geqslant 1$ then $|a - b|^p <
a^p + b^p$), in order to conclude that $g_u (z)$ is strictly positive on
$\overline{B (0, s)}$, it is sufficient to prove that $u (\cdot)$ and $u (\cdot
+ z)$ are {{\em both strictly positive\/}} on a set of positive Lebesgue
measure, namely, as we now show, on $B (0, \rho)$. Indeed, after that, we have
\[ \int_{\R^N} |u (x + z) - u (x) |^p \hspace{0.17em} \mathrm{d} x <
   \int_{\R^N} u^p (x + z) + u^p (x) \hspace{0.17em} \mathrm{d} x = 2 \|u\|_{L^p}^p
   \quad \text{for all } |z| \leqslant s, \]
that is $g_u (z) > 0$ for every $|z| \leqslant s$, from which \eqref{eq:minu}
follows. Thus, we have to prove that for every $| z | \leqslant s$ both $u
(\cdot)$ and $u (\cdot + z)$ are {{\em strictly positive\/}} on $B (0, \rho)$.
This is indeed the case. For every $|z| \leqslant s$ and every $| x | < \rho$
we have
\[ | x | < r, \quad | x + z | < r. \]
Hence $u (x) > 0$ and $u (x + z) > 0$. Thus, for every $|z| \leqslant s$, the
functions $u (\cdot)$ and $u (\cdot + z)$ are {{\em strictly positive\/}} on
the open set $B (0, \rho)$ of positive Lebesgue measure.

Applying \eqref{eq:JMS2} to the test function $u$ in \eqref{eq:testuprops},
and using \eqref{eq:minu}, we get that
\[ 0 = \lim_{\varepsilon \to 0}  \int_{\R^N} \rho_{\varepsilon} (z)
   g_u (z) \hspace{0.17em} \mathrm{d} z \geqslant \limsup_{\varepsilon \to 0}  \int_{|z|
   \leqslant s} \rho_{\varepsilon} (z) g_u (z) \hspace{0.17em} \mathrm{d} z \geqslant m_u (s)
   \limsup_{\varepsilon \to 0}  \int_{|z| \leqslant s} \rho_{\varepsilon} (z)
   \hspace{0.17em} \mathrm{d} z. \]
Since $m_u (s) > 0$ we deduce
\[ \lim_{\varepsilon \to 0}  \int_{|z| \leqslant s} \rho_{\varepsilon} (z)
   \hspace{0.17em} \mathrm{d} z = 0. \]
Because $s > R$ and $\rho_{\varepsilon} \geqslant 0$, we also get
\[ 0 \leqslant \limsup_{\varepsilon \to 0}  \int_{|z| \leqslant R}
   \rho_{\varepsilon} (z) \hspace{0.17em} \mathrm{d} z \leqslant \lim_{\varepsilon \to 0} 
   \int_{|z| \leqslant s} \rho_{\varepsilon} (z) \hspace{0.17em} \mathrm{d} z = 0, \]
which proves the desired mass escape
\[ \lim_{\varepsilon \to 0}  \int_{|z| \leqslant R} \rho_{\varepsilon} (z)
   \hspace{0.17em} \mathrm{d}z = 0, \quad \text{and hence} \quad \lim_{\varepsilon \to 0} 
   \int_{|z| > R} \rho_{\varepsilon} (z) \hspace{0.17em} \mathrm{d} z = 1. \]

\smallskip
{\noindent}\text{{\bfseries{Step 2: Short-range attenuation effect.}}} Let $R
> 0$ be fixed. Choose a function $u \in C^{\infty}_c (\R^N)$ with
$\supp u \subset B (0, R / 2)$. For any $|z| > R$, the supports of $u
(\cdot + z)$ and $u (\cdot)$ are disjoint, so $A (z) = 2 \|u\|_{L^p}^p$ for
any $| z | > R$. It follows that
\[ \mathcal{F}_{\varepsilon} (u) = 2 \|u\|_{L^p}^p  \int_{|z| > R}
   \rho_{\varepsilon} (z) \hspace{0.17em} \mathrm{d} z + \int_{|z| \leqslant R}
   \rho_{\varepsilon} (z) A (z) \hspace{0.17em} \mathrm{d} z. \]
We take the limit as $\varepsilon \to 0$. By the MS formula \eqref{eq:MS-smooth} and the mass escape property from Step 1, we get
\[ \lim_{\varepsilon \to 0}  \int_{|z| \leqslant R} \rho_{\varepsilon} (z) A
   (z) \hspace{0.17em} \mathrm{d} z = 0. \]
By Lemma \eqref{lemma:(ii)-smooth} we conclude.
\end{proof}

We omit the proofs of the next two corollaries, for they are a direct consequence of the proof of Theorem~\ref{thm:MS-smooth}. 
\begin{corollary}
Assume that the kernels $\{\rho_{\eps}\}$ further satisfy
    \begin{equation*}
      \lim_{R \to \infty} \lim_{\varepsilon \to 0} 
      \int_{|z| > R} \rho_{\varepsilon} (z) \hspace{0.17em} \d z = 1.
    \end{equation*}
    Then,
    \begin{equation*}
        \lim_{R \to \infty} \lim_{\varepsilon \to 0} 
      \int_{|z| < R} |z|^p \rho_{\varepsilon} (z) \hspace{0.17em} \d z
      = 0
    \end{equation*}
    if and only if for every smooth, compactly supported function $u
    \in C^{\infty}_c (\R^N)$, $\mathcal{F}_{\varepsilon} (u)$ is well-defined    for $\varepsilon$ sufficiently small, and
    \begin{equation*}
      \lim_{\varepsilon \to 0} \mathcal{F}_{\varepsilon}
      (u) = 2 \|u\|^p_{L^p (\R^N)} .
    \end{equation*}
\end{corollary}
\begin{corollary}
Assume that the kernels $\{\rho_{\eps}\}$ further satisfy
    \begin{equation*}
      \lim_{R \to \infty} \lim_{\varepsilon \to 0} 
      \int_{|z| < R} |z|^p \rho_{\varepsilon} (z) \hspace{0.17em} \d z
      = 0.
    \end{equation*}
    Then,
    \begin{equation*}
    \lim_{R \to \infty} \lim_{\varepsilon \to 0} 
      \int_{|z| > R} \rho_{\varepsilon} (z) \hspace{0.17em} \d z = 1,
    \end{equation*}
    if and only if for every smooth, compactly supported function $u
    \in C^{\infty}_c (\R^N)$, $\mathcal{F}_{\varepsilon} (u)$ is well-defined
    for $\varepsilon$ sufficiently small, and
    \begin{equation*}
      \lim_{\varepsilon \to 0} \mathcal{F}_{\varepsilon}
      (u) = 2 \|u\|^p_{L^p (\R^N)} .
    \end{equation*}
\end{corollary}

\color{black}

\subsection{Corollary: The MS formula in Sobolev spaces of integer
order}\label{subsec2.2}In Theorem~\ref{thm:MS-smooth}, we have generalized the
Maz'ya--Shaposhnikova identity \eqref{eq:MS-smooth} to a significantly wider
family of kernels by working within the space $C_c^{\infty} (\R^N)$. Although this smooth, compactly supported setting strengthens the implication
(\ref{eq:MS+II-Sob}~$\Rightarrow$~\ref{cond:R-fix-Sob}), it appears at first
glance, to weaken the converse implication
(\ref{cond:R-fix-Sob}~$\Rightarrow$~\ref{eq:MS+II-Sob}). In the present
section, we demonstrate that no generality is lost under this restriction: one
may equivalently replace \ $C_c^{\infty} (\R^N)$ with the Sobolev space $W^{1,
p} (\R^N)$, again, under the broader hypotheses on the kernels stated in
equation \eqref{eq:kernels-R_fix}. This extension is a direct consequence of the well‑posedness of the corresponding energy functionals established in
Remark~\ref{lemma:inclusion-smooth}, combined with Theorem~\ref{thm:MS-smooth}.
In Section~\ref{sec:FS}, we will further show that, for the special
case of fractional kernels, our main result recovers the classical
Maz'ya--Shaposhnikova formula for the fractional Sobolev space $W^{s, p}
(\R^N)$.

\begin{corollary}[Sobolev setting]
  \label{cor:MS-sobolev} The following three conditions are equivalent:
  \begin{enumerate}
    \item \label{cond:R-fix-Sob}{\textbf{{\emph{Uniform moment
    conditions}}}}. For every fixed radius $R > 0$, as $\varepsilon
    \to 0$ the kernels $\{ \rho_{\varepsilon} \}_{\varepsilon > 0}$
    satisfy
    \begin{equation}
      \label{eq:kernels-R_fix-Sob} \lim_{\varepsilon \to 0}  \int_{|z| > R}
      \rho_{\varepsilon} (z) \hspace{0.17em} \d z = 1 \quad \text{and}
      \quad \lim_{\varepsilon \to 0}  \int_{|z| < R} |z|^p \rho_{\varepsilon}
      (z) \hspace{0.17em} \d z = 0.
    \end{equation}
    \item {\emph{{\textbf{\label{cond:R-lim-Sob}Iterated limits
    conditions.}}}} The double limits in the order ``first $\varepsilon
    \to 0$, then $R \rightarrow \infty$'' satisfy
    \begin{equation}
      \label{eq:kernels-R_lim-Sob} \lim_{R \to \infty} \lim_{\varepsilon \to
      0}  \int_{|z| > R} \rho_{\varepsilon} (z) \hspace{0.17em} \d z =
      1 \quad \text{and} \quad \lim_{R \to \infty} \lim_{\varepsilon \to 0} 
      \int_{|z| < R} |z|^p \rho_{\varepsilon} (z) \hspace{0.17em} \d z
      = 0.
    \end{equation}
    \item \label{eq:MS+II-Sob}{\emph{{\textbf{Maz'ya--Shaposhnikova formula
    in $W^{1, p} (\R^N)$.}} }}For every $u \in W^{1, p} (\R^N)$, the energy
    $\mathcal{F}_{\varepsilon} (u)$ is well-defined for $\varepsilon > 0$
    sufficiently small, and
    \begin{equation}
      \label{eq:MS-Sob} \lim_{\varepsilon \to 0} \mathcal{F}_{\varepsilon} (u)
      = 2 \|u\|^p_{L^p (\R^N)} .
    \end{equation}
    Moreover, for each $R > 0$, there holds
    \begin{equation}
      \lim_{\varepsilon \to 0}  \int_{|z| < R} \int_{\R^N}
      \rho_{\varepsilon} (z)  |u (x + z) - u (x) |^p \hspace{0.17em}
      \d x \hspace{0.17em} \d z = 0.
    \end{equation}
  \end{enumerate}
\end{corollary}

\begin{proof}
  {\textbf{(\ref{cond:R-fix-Sob} $\Rightarrow$ \ref{cond:R-lim-Sob}).}}
  Trivial by taking the limits in the prescribed order.{\smallskip}
  
  {\noindent}{\textbf{(\ref{cond:R-lim-Sob} $\Rightarrow$
  \ref{eq:MS+II-Sob})}} The well‑posedness of the associated energy
  functionals established in Remark~\ref{lemma:inclusion-smooth} immediately
  yields the first conclusion in item~\ref{eq:MS+II-Sob}. To prove the
  remaining statements, we first note that, by the monotonicity in $R > 0$, the
  second limit relation in equation \eqref{eq:kernels-R_lim-Sob} is equivalent
  to
  \begin{equation}
    \lim_{\varepsilon \to 0}  \int_{|z| < R} |z|^p \rho_{\varepsilon} (z)
    \hspace{0.17em} \d z = 0 \quad \text{for every } R > 0.
  \end{equation}
  Let $u \in W^{1, p} (\R^N)$ and decompose the nonlocal energy as in
  \eqref{eq:decomp}, i.e., $\mathcal{F}_{\veps} (u) = I_{\veps, R} [u] + I 
  I_{\veps, R} [u]$. To estimate $\II_{\veps, R} [u]$, we observe that for
  fixed $R > 0$ there holds
  \begin{align}
    \lim_{\varepsilon \to 0} \II_{\varepsilon, R} [u] & = 
    \lim_{\varepsilon \to 0}  \int_{|z| < R} \int_{\R^N}
    \rho_{\varepsilon} (z)  |u (x + z) - u (x) |^p \hspace{0.17em} \d
    x \hspace{0.17em} \d z \nonumber\\
    & \leqslant  \| \nabla u\|^p_{L^p (\R^N)} \lim_{\varepsilon
    \to 0}  \int_{|z| < R} |z|^p \rho_{\varepsilon} (z)
    \hspace{0.17em} \d z = 0.  \label{eq:II-Sob}
  \end{align}
  Thus, 
\begin{align}
  \limsup_{\varepsilon \rightarrow 0} \mathcal{F}_{\varepsilon} (u) & =
  \lim_{R \rightarrow \infty} \limsup_{\varepsilon \rightarrow 0}
  I_{\varepsilon, R} [u],\\
  \liminf_{\varepsilon \rightarrow 0} \mathcal{F}_{\varepsilon} (u) & =
  \lim_{R \rightarrow \infty} \liminf_{\varepsilon \rightarrow 0}
  I_{\varepsilon, R} [u] . 
\end{align}
  To estimate $I_{\veps, R} [u]$ we argue by density. Let $(u_n) \in
  C_c^{\infty} (\R^N)$ be a sequence such that $u_n \rightarrow u$ in
  $L^p (\R^N)$ as $n \to \infty$. 
  
  We first notice that for every $\tau>0$ there holds (see \cite[Lemma 2]{DiFio2020bmo})
\begin{align}
  I_{\varepsilon, R} [u] & \leqslant  (1 + \tau)^{p-1} I_{\varepsilon, R} [u_n] +
  \left(\frac{1 + \tau}{\tau}\right)^{p-1} I_{\varepsilon, R} [u - u_n],  \label{eq:step1}\\
  I_{\varepsilon, R} [u] & \geqslant  \frac{1}{(1 + \tau)^{p-1}} I_{\varepsilon, R}
  [u_n] - \frac{1}{\tau^{p-1}} I_{\varepsilon, R} [u - u_n] .  \label{eq:step1b}
\end{align}
  We observe that
\begin{align}
  I_{\varepsilon, R} [u - u_n] & =  \int_{|z| > R}
  \rho_{\varepsilon} (z)  \int_{x \in \R^N} |u (x + z) - u_n (x + z) +
  u_n (x) - u (x) |^p \d  x\,\d  z \nonumber\\
  & \leqslant 2^{p-1}\int_{|z| > R} \rho_{\varepsilon} (z)  \int_{x \in
  \R^N} |u (x + z) - u_n (x + z) |^p \d  x\,\d  z \nonumber\\
  &  \qquad \qquad \qquad \qquad + 2^{p-1}\int_{|z| > R} \rho_{\varepsilon} (z) 
  \int_{x \in \R^N} |u (x) - u_n (x) |^p \d  x\,\d  z \nonumber\\
  & =  2^p \| u - u_n \|^p_{L^p} \int_{|z| > R} \rho_{\varepsilon} (z)\, \d z . 
  \label{eq:step2}
\end{align}
  Proceeding as in the proof of
  Theorem~\ref{thm:MS-smooth} (cf. \eqref{eq:I-smooth}), we find that for every $n \in \N$,
  \begin{equation}
  \label{eq:step3} I_{\varepsilon, R} (u_n) = 2 \|u_n
     \|_{L^p}^p   \int_{|z| > R_n}
     \rho_{\varepsilon} (z) \hspace{0.17em} \d z, 
     \end{equation}
  for every $R> 2 R_n$, where the latter radius is chosen so that $\supp u_n \subseteq B_{R_n}$.

By \eqref{eq:step1}--\eqref{eq:step1b}, together with \eqref{eq:step2} and \eqref{eq:step3}, we obtain that the following estimates hold
\begin{align}
  I_{\varepsilon, R} [u] & \leqslant  \left( 2(1 + \tau)^{p-1} \|u_n \|^p_{L^p} + 2^p
  \left(\frac{1 + \tau}{\tau}\right)^{p-1} \| u - u_n \|^p_{L^p} \right) \int_{|z| > R}
  \rho_{\varepsilon} (z)\, \d  z,\\
  I_{\varepsilon, R} [u] & \geqslant  \left( \frac{2}{(1 + \tau)^{p-1}} \|u_n
  \|^p_{L^p} - \frac{2^p}{\tau^{p-1}} \| u - u_n \|^p_{L^p} \right) \int_{|z| > R}
  \rho_{\varepsilon} (z) \, \d  z, 
\end{align}
for every $R>2R_n$.
Assume now that $n$ is big enough, so that 
\begin{equation}
\left( \frac{2}{(1 + \tau)^{p-1}} \|u_n
  \|^p_{L^p} - \frac{2^p}{\tau^{p-1}} \| u - u_n \|^p_{L^p} \right)\geqslant 0.
\end{equation}
From the previous two relations, we obtain
\begin{align}
  \limsup_{\varepsilon \rightarrow 0} I_{\varepsilon, R} [u] & \leqslant 
  \left( 2(1 + \tau)^{p-1} \|u_n \|^p_{L^p} + 2^p
  \left(\frac{1 + \tau}{\tau}\right)^{p-1}  \| u - u_n
  \|^p_{L^p} \right)  \nonumber\\
  &\qquad \qquad \qquad \qquad \qquad \qquad\cdot \limsup_{\varepsilon \rightarrow 0}  \int_{|z| > R} \rho_{\varepsilon} (z) \, \d  z, \\
  \liminf_{\varepsilon \rightarrow 0} I_{\varepsilon, R} [u] & \geqslant 
   \left( \frac{2}{(1 + \tau)^{p-1}} \|u_n
  \|^p_{L^p} - \frac{2^p}{\tau^{p-1}} \| u - u_n \|^p_{L^p} \right) \nonumber\\
  &\qquad \qquad \qquad \qquad \qquad \qquad\cdot \liminf_{\varepsilon \rightarrow 0}
  \int_{|z| > R} \rho_{\varepsilon} (z)\, \d  z, 
\end{align}
for every $R>2R_n$.
Taking the limit for $R \rightarrow \infty$, we infer
\begin{align}\nonumber
  \limsup_{R\to \infty}\limsup_{\varepsilon \rightarrow 0} I_{\varepsilon, R} [u]  
  &\leqslant 
  \left( 2(1 + \tau)^{p-1} \|u_n \|^p_{L^p} + 2^p
  \left(\frac{1 + \tau}{\tau}\right)^{p-1}  \| u - u_n
  \|^p_{L^p} \right) \\
  &\qquad \qquad \qquad \qquad \qquad \qquad \cdot \limsup_{R\to \infty}\limsup_{\varepsilon \rightarrow 0}  \int_{|z| > R} \rho_{\varepsilon} (z) \, \d  z, \\
  \liminf_{R\to \infty}\liminf_{\varepsilon \rightarrow 0} I_{\varepsilon, R} [u] & \geqslant 
   \left( \frac{2}{(1 + \tau)^{p-1}} \|u_n
  \|^p_{L^p} - \frac{2^p}{\tau^{p-1}} \| u - u_n \|^p_{L^p} \right)  \nonumber\\
   &\qquad \qquad \qquad \qquad \qquad \qquad\cdot  \liminf_{R\to \infty}\liminf_{\varepsilon \rightarrow 0}
  \int_{|z| > R} \rho_{\varepsilon} (z)\, \d z.
\end{align}
Thus, passing first to the limit for $n \rightarrow \infty$ and then for $\tau
\rightarrow 0$ we conclude
\begin{align}
  \label{eq:need1} \limsup_{R \rightarrow \infty} \limsup_{\varepsilon \rightarrow 0}
  I_{\varepsilon, R} [u] & \leqslant  2 \|u\|^p_{L^p}\limsup_{R\to \infty}\limsup_{\varepsilon \rightarrow 0}  \int_{|z| > R} \rho_{\varepsilon} (z) \, \d  z, \\
  \label{eq:need2}\liminf_{R \rightarrow \infty} \liminf_{\varepsilon \rightarrow 0}
  I_{\varepsilon, R} [u] & \geqslant   2 \|u  \|^p_{L^p}\liminf_{R\to \infty}\liminf_{\varepsilon \rightarrow 0}  \int_{|z| > R} \rho_{\varepsilon} (z) \, \d  z,
\end{align}
which, in turn, by \eqref{eq:kernels-R_lim-Sob} yields 
\begin{equation}
\label{eq:I-Sob}
\lim_{R\to +\infty}\lim_{\eps\to 0} I_{\eps,R}[u]= 2 \|u  \|^p_{L^p}.
\end{equation}
  Finally, combining \eqref{eq:II-Sob} and \eqref{eq:I-Sob}, we infer that
\begin{equation} 
\lim_{\varepsilon \to 0} \mathcal{F}_{\varepsilon} (u) = 2
     \|u\|_{L^p}^p \lim_{R \rightarrow \infty} \lim_{\varepsilon \rightarrow
     0}  \int_{|z| > R} \rho_{\varepsilon} (z) \hspace{0.17em} \d z =
     2 \|u\|_{L^p}^p .
\end{equation}
  This proves the implication (\ref{cond:R-lim-Sob} $\Rightarrow$
  \ref{eq:MS+II-Sob}).{\smallskip}
  
  {\noindent}{\textbf{(\ref{eq:MS+II-Sob} $\Rightarrow$
  \ref{cond:R-fix-Sob}).}} Since relation \eqref{eq:MS-Sob} holds in
  particular for any $u \in C^{\infty}_c (\R^N)$, the implication
  (\ref{eq:MS+II-Sob} $\Rightarrow$ \ref{cond:R-fix-Sob}) proved in
  Theorem~\ref{thm:MS-smooth} applies. This completes the proof.
\end{proof}

\section{The case of fractional Sobolev spaces}\label{sec:FS}

In this section, we analyze nonlocal energy functionals acting on fractional
Sobolev spaces $W^{s, p} (\R^N)$ with exponent $s \in (0, 1)$.
Unlike the integral‑order case, the weaker regularity of functions in $W^{s,
p} (\R^N)$ requires a corresponding adaptation of the hypotheses
on the kernel family $\{ \rho_{\varepsilon} \}_{\varepsilon > 0}$. Under these
modified assumptions, one obtains a direct analog of the classical
Maz'ya--Shaposhnikova convergence formula, now valid in the fractional regime.

\begin{theorem}[Fractional Maz'ya--Shaposhnikova formula]
  \label{thm:MS-frac}Let $p \in [1, \infty)$ and $s \in (0, 1)$. Suppose the
  kernels $\rho_{\varepsilon} : \R^N \rightarrow [0, \infty)$ satisfy
  \begin{equation}
    \lim_{R \to \infty} \lim_{\varepsilon \to 0}  \int_{|z| > R}
    \rho_{\varepsilon} (z) \hspace{0.17em} \d z = 1 \quad \text{and}
    \quad \lim_{R \to \infty} \lim_{\varepsilon \to 0}  \int_{|z| < R}
    |z|^{sp} \rho_{\varepsilon} (z) \hspace{0.17em} \d z = 0.
    \label{eq:kernels-R_lim-frac}
  \end{equation}
  Then, for every $u \in W^{s, p} (\R^N)$,
  \begin{equation}
    \label{eq:MS-frac} \lim_{\varepsilon \to 0} \mathcal{F}_{\varepsilon} (u)
    = 2 \|u\|^p_{L^p (\R^N)} .
  \end{equation}
  Moreover, for each $R > 0$,
  \begin{equation}
    \label{eq:II-frac} \lim_{\varepsilon \to 0} II_{\varepsilon} [u] =
    \lim_{\varepsilon \to 0}  \int_{|z| < R} \int_{\R^N}
    \rho_{\varepsilon} (z)  |u (x + z) - u (x) |^p \hspace{0.17em} \d
    x \hspace{0.17em} \d z = 0.
  \end{equation}
\end{theorem}

Before proceeding with the proof we make some remarks.
\begin{remark} 
Note that, as we observed in Theorem~\ref{thm:MS-smooth}, by monotonicity in 
$R > 0$, the second condition in \eqref{eq:kernels-R_lim-frac} is equivalent 
to 
\[ \lim_{\varepsilon \to 0} \int_{|z| < R} |z|^{sp} \rho_{\varepsilon} (z) 
\hspace{0.17em} \d z = 0 \quad \text{for every } R > 0. \] 
However, no similar simplification applies to the first condition: although 
\begin{equation} 
\lim_{\varepsilon \to 0} \int_{|z| > R} \rho_{\varepsilon} (z) 
\hspace{0.17em} \d z = 1 \quad \text{for every } R > 0 
\label{eq:tempj} 
\end{equation} 
 implies the first limit in \eqref{eq:kernels-R_lim-frac}, 
the converse need not hold a priori.
\end{remark}
\begin{remark}[An open question]
  \label{rmk:genuinegen} The classical family of {\emph{fractional}} kernels
  \begin{equation}
    \label{eq:frac-kernel} \rho_{\varepsilon} (z) = \frac{\varepsilon p}{|
    \mathbb{S}^{N - 1} |} |z|^{- (N + \varepsilon p)}  \quad \text{with }
    \veps \in (0, 1)
  \end{equation}
  satisfies condition \eqref{eq:kernels-R_lim-frac}. Indeed, for each $s \in (0, 1)$, one checks directly that if $\varepsilon
  \leqslant s$, then as $\varepsilon \to 0$ both the limit relations
  in \eqref{eq:kernels-R_lim-frac} are satisfied. Thus,
  Theorem~\ref{thm:MS-frac} indeed recovers the classical
  Maz'ya--Shaposhnikova convergence formula while allowing any kernel family
  satisfying the more abstract conditions \eqref{eq:kernels-R_lim-frac}. In
  this sense, our Theorem~\ref{thm:MS-frac} is a genuine generalization of the
  fractional Maz'ya--Shaposhnikova formula.
  
  Moreover, when $u$ has higher regularity---either $u \in C^{\infty}$
  (Theorem~\ref{thm:MS-smooth}) or $u \in W^{1, p}$
  (Corollary~\ref{cor:MS-sobolev})---one can accommodate kernels that are even
  {\emph{more singular}} at the origin. This reflects the fact that greater
  function smoothness permits weaker tail assumptions on $\{\rho_{\varepsilon}\}_{\varepsilon>0}$.
  
  Despite this generality, the fractional result lacks the full equivalence
  found in the integral‑order case. 
  It remains an open question whether the condition 
  \[ \lim_{\varepsilon \to 0}
     {II}_{\varepsilon, R} [u] = 0 \quad \forall \text{} R > 0,
     \, \forall u \in W^{s, p}(\R^N), \]
  already implies the second requirement in \eqref{eq:kernels-R_lim-frac} without placing further restrictions on the class of kernels.  A positive answer would provide a complete converse to Theorem~\ref{thm:MS-frac}, mirroring the equivalences of Theorem~\ref{thm:MS-smooth}. In Theorem~\ref{3-Thm2} below, we take a first step in that direction by proving the converse under additional hypotheses on the admissible kernels; our partial converse for admissible kernels pinpoints where those difficulties arise and indicates natural directions for further work.
\end{remark}

In the following lemma, we verify that the nonlocal energies
are well-defined on $W^{s, p} (\R^N)$ under \eqref{eq:kernels-R_lim-frac}.

\begin{lemma}[Well‑Definedness of the Nonlocal Energies in $W^{s, p}$]
  \label{lemma:inclusion-frac} Let $p \in [1,  \infty)$ and $s \in (0, 1)$.
  Assume that the kernels $\{ \rho_{\varepsilon} \}$ satisfy the two limit
  relations in {\emph{\eqref{eq:kernels-R_lim-frac}}}. Then, there exists
  $\veps_0 > 0$ such that, for all $0 < \varepsilon < \varepsilon_0$ the
  continuous embedding {\emph{(cf.\ \eqref{f-space})}}
  \begin{equation}
    \label{eq:coninclfrac} W^{s, p} (\R^N) \hookrightarrow X^p_{\eps_0} (\R^N)
   = \bigcap_{\eps \leqslant \eps_0} \mathcal{X}^p_{\eps} (\R^N) 
  \end{equation}
  holds. More precisely, there is a constant $c_{\rho} > 0$, depending only on $\{
  \rho_{\varepsilon} \}$, such that for all $0 < \varepsilon \leqslant
  \varepsilon_0$
  \[ \mathcal{F}_{\varepsilon} (u) \leqslant c_{\rho} \|u\|^p_{W^{s, p}} \quad
     \forall \, u \in W^{s, p} (\R^N) . \]
\end{lemma}

\begin{proof}
  Let us denote by $[u]^p_{W^{s, p}}$ the fractional seminorm, defined by
  \begin{equation}
    [u]^p_{W^{s, p}}:= \iint_{\R^N \times \R^N} \frac{|u
    (x) - u (y) |^p}{|x - y|^{N + sp}} \hspace{0.17em} \d x
    \hspace{0.17em} \d y
  \end{equation}
  and by $\| u \|_{W^{s, p}}:= \| u \|_{L^p} + [u]_{W^{s, p}}$ the
  associated fractional norm.
  
  We divide the proof into two steps.{\smallskip}
  
  {\noindent}{\textbf{Step~1.}}
  First, we show that for every $z \in \R^N$, there holds
  \begin{equation}
    \label{ineq-frac} \| \tau_z u - u\|_{L^p} \leqslant c |z|^s \| u \|_{W^{s,
    p}}
  \end{equation}
  for some positive constant $c > 0$ that depends only on $s, p$. For that,
  recall (see e.g., {\cite[Lemma 6.14]{L23}}) that there exists a constant $c
  > 0$, that depends only on $s, p$, such that for all $0 < |z| < 1 / 2$ there
  holds
  \begin{equation}
    \label{ineq-fracold1} \| \tau_z u - u\|_{L^p} \leqslant c |z|^s 
    [u]_{W^{s, p}} .
  \end{equation}
  Now, if $| z | \geqslant 1 / 2$ then $(2 | z |)^s \geqslant 1$ and,
  therefore,
  \begin{equation}
    \| \tau_z u - u \|_{L^p} \, \leqslant 2 \text{} \| u \|_{L^p}
    \leqslant 2^{s+1}  | z |^s \| u \|_{L^p} \leqslant 2^{s+1}  | z |^s \| u \|_{W^{s,
    p}} . \label{ineq-fracold2}
  \end{equation}
  Combining \eqref{ineq-fracold1} and \eqref{ineq-fracold2} we get
  \eqref{ineq-frac}.{\smallskip}
  
  {\noindent}{\textbf{Step~2.}} We now show \eqref{eq:coninclfrac}. By
  assumptions \eqref{eq:kernels-R_lim-frac}, there exists some $\varepsilon_0
  > 0$ such that for every $\varepsilon \leqslant \varepsilon_0$
  \[ \int_{\R^N} \min \{1, |z|^{sp} \} \rho_{\varepsilon} (z) \hspace{0.17em}
     \d z < + \infty . \]
  Let $u \in W^{s, p} (\R^N)$. Then, for every $z \in \R^N$, using
  \eqref{ineq-frac}, the following inequalities hold
  \[ \| \tau_z u - u\|_{L^p} \leqslant 2 \|u\|_{L^p (\R^N)} \leqslant (2 + c)
     \| u \|_{W^{s, p}}, \]
  \[ \| \tau_z u - u\|_{L^p} \leqslant (2 + c)  |z|^s \| u \|_{W^{s, p}} .
  \]
  Therefore, combining the two previous estimates, we get
  \[ \| \tau_z u - u\|_{L^p} \leqslant (2 + c) \min \{1, |z|^s \} \|u\|_{W^{s,
     p}} . \]
  Elevating both sides to the power $p$ and integrating in $z$ against
  $\rho_{\varepsilon} (z) \d z$, we obtain that
  \[ \mathcal{F}_{\varepsilon} (u) \leqslant \|u\|^p_{W^{s, p}}  \int_{\R^N}
     \min \{1, |z|^{sp} \} \rho_{\varepsilon} (z) \hspace{0.17em} \d z
     < + \infty \]
  for every $\varepsilon \leqslant \varepsilon_0$. This proves that $u \in
  \mathcal{X}^p_{\varepsilon} (\R^N)$ for every $\varepsilon \leqslant
  \varepsilon_0$.
\end{proof}

\subsection{Proof of Theorem~\ref{thm:MS-frac}}In this subsection, we prove Theorem~\ref{thm:MS-frac}.

By monotonicity in $R > 0$, the second condition in
\eqref{eq:kernels-R_lim-frac} is equivalent to
\begin{equation}
  \label{eq:cond-equiv} \lim_{\varepsilon \to 0}  \int_{|z| < R} | z |^{sp}
  \rho_{\varepsilon} (z) \hspace{0.17em} \d z = 0 \quad \text{for
  every } R > 0.
\end{equation}
Fix $u \in W^{s, p} (\R^N)$ and any $R > 0$. Decompose the nonlocal energy as
in \eqref{eq:decomp}, i.e., $\mathcal{F}_{\veps} (u) = I_{\veps, R} [u] + I 
I_{\veps, R} [u]$. Recalling the translation‑difference estimate
\eqref{ineq-frac}, we have that
\begin{eqnarray}
  \lim_{\varepsilon \to 0} \II_{\varepsilon, R} [u] & = &
  \lim_{\varepsilon \to 0}  \int_{|z| < R} \int_{\R^N}
  \rho_{\varepsilon} (z)  |u (x + z) - u (x) |^p \hspace{0.17em} \d x
  \hspace{0.17em} \d z \nonumber\\
  & \leqslant & c \| u \|_{W^{s, p}}^p \lim_{\varepsilon \to 0} 
  \int_{|z| < R} |z|^{sp} \rho_{\varepsilon} (z) \hspace{0.17em} \d z
  = 0.  \label{eq:II-fractional}
\end{eqnarray}
Thus, 
\begin{align}
  \limsup_{\varepsilon \rightarrow 0} \mathcal{F}_{\varepsilon} (u) & =
  \lim_{R \rightarrow \infty} \limsup_{\varepsilon \rightarrow 0}
  I_{\varepsilon, R} [u], \nonumber\\
  \liminf_{\varepsilon \rightarrow 0} \mathcal{F}_{\varepsilon} (u) & =
  \lim_{R \rightarrow \infty} \liminf_{\varepsilon \rightarrow 0}
  I_{\varepsilon, R} [u] . \nonumber
\end{align}
  Next, to estimate $I_{\veps, R} [u]$ we argue by density as in the proof of Corollary \ref{cor:MS-sobolev}. Let $(u_n) \in C_c^{\infty} (\R^N)$ be a sequence such that $u_n
\rightarrow u$ in $W^{s, p} (\R^N)$ as $n \to \infty$, (see, e.g.,
{\cite[Theorem 6.66]{L23}}). We have that for every $\tau>0$ there holds (see again \cite[Lemma 2]{DiFio2020bmo})
\begin{eqnarray}
  I_{\varepsilon, R} [u] & \leqslant & (1 + \tau)^{p-1} I_{\varepsilon, R} [u_n] +
  \left(\frac{1 + \tau}{\tau}\right)^{p-1} I_{\varepsilon, R} [u - u_n],  \label{eq:step1-frac}\\
  I_{\varepsilon, R} [u] & \geqslant & \frac{1}{(1 + \tau)^{p-1}} I_{\varepsilon, R}
  [u_n] - \frac{1}{\tau^{p-1}} I_{\varepsilon, R} [u - u_n] .  \label{eq:step1b-frac}
\end{eqnarray}

The conclusion then follows by arguing exactly as in the proof of
Corollary~\ref{cor:MS-sobolev}. 

\subsection{Towards necessary conditions in the fractional Sobolev-spaces setting}
In this final subsection, we show that, under an extra hypothesis on the
admissible kernels $\{\rho_{\varepsilon} \}$, the conditions introduced in Theorem~\ref{thm:MS-frac} are not only sufficient but also necessary for the
validity of a Maz'ya--Shaposhnikova type formula.

We begin by defining the class of kernels that we will consider.

\begin{definition}[Admissible families of kernels]
  \label{classA} A family of non-negative measurable kernels $\{\rho_{\varepsilon}
\}_{\varepsilon > 0}$ belongs to the class $\mathcal{A}_{s, p}$, for $p > 1$
and $0 < s < 1$, if the following normalization condition holds:
\[ \lim_{\varepsilon \to 0}  \int_{\R^N} (1 \wedge |z|^{sp})
   \rho_{\varepsilon} (z) \hspace{0.17em} \mathrm{d} z = 1, \]
where $1 \wedge |z|^{sp} = \min (1, |z|^{sp})$.

A family of kernels $\{\rho_{\varepsilon} \}_{\varepsilon > 0}$ satisfies the
\emph{uniform moment conditions} if for every fixed radius $R > 0$, the following two
limits hold as $\varepsilon \to 0$:
\begin{enumerate}
  \item \emph{Mass Escape:}
  \begin{equation} \label{eq:massescape}
   \lim_{\varepsilon \to 0}  \int_{|z| > R}
  \rho_{\varepsilon} (z) \hspace{0.17em} \mathrm{d} z = 1.
\end{equation} 
  \item \emph{Short-Range Attenuation:}
    \begin{equation} \label{eq:sratten}
   \lim_{\varepsilon \to 0}  \int_{|z|
  < R} |z|^{sp} \rho_{\varepsilon} (z) \hspace{0.17em} \mathrm{d} z = 0.
\end{equation} 
\end{enumerate}
\end{definition}

\begin{remark}
  Fractional-type kernels of the form
  \[ \rho_{\varepsilon} (z) := \frac{\varepsilon p}{|\mathbb{S}^{N - 1} |}
     \hspace{0.17em} |z|^{- (N + \varepsilon p)} \]
   belong to $\mathcal{A}_{s, p}$ for every
  $0 < s < 1$ and every $p > 1$. Indeed, 
  we compute
  \[ \frac{\varepsilon p}{|\mathbb{S}^{N - 1} |}  \int_{\mathbb{R}^N} \frac{1
     \wedge |z|^{sp}}{|z|^{N + \varepsilon p}} \hspace{0.17em} \mathrm{d} z =
     \frac{s}{s - \varepsilon} \xrightarrow{\varepsilon \to 0} 1. \]
\end{remark}

The main result of this section is Theorem~\ref{3-Thm2}, stated below. In it, we show that when the analysis is restricted to kernels in $\mathcal{A}_{s,p}$, the implication in Theorem~\ref{thm:MS-frac} can be sharpened to an equivalence.
 This gives a precise characterization of those kernels in $\mathcal{A}_{s, p}$ for which the MS formula holds: they are exactly the kernels that satisfy the uniform moment conditions. Before moving further in this direction, let us clarify the relationships among the concepts introduced above.
Our first observation is that the uniform moment conditions are strictly stronger than simple membership in the admissible class $\mathcal{A}_{s,p}$.
The following result makes this statement precise:

\begin{lemma}
  If a family of kernels $\{\rho_{\varepsilon} \}_{\varepsilon > 0}$ satisfies
  the two uniform moment conditions, then it necessarily belongs to the class
  of admissible fractional kernels $\mathcal{A}_{s, p}$.
\end{lemma}

\begin{proof}
To establish membership in $\mathcal{A}_{s, p}$, we
must evaluate the limit of the admissibility integral. We decompose this
integral over the regions $|z| < 1$ and $|z| \geqslant 1$:
\[ \int_{\R^N} (1 \wedge |z|^{sp}) \rho_{\varepsilon} (z) \hspace{0.17em} \mathrm{d} z =
   \int_{|z| < 1} |z|^{sp} \rho_{\varepsilon} (z) \hspace{0.17em} \mathrm{d} z + \int_{|z| \geqslant
   1} \rho_{\varepsilon} (z) \hspace{0.17em} \mathrm{d} z. \]
We take the limit as $\varepsilon \to 0$ and apply the two moment conditions
with the specific radius $R = 1$. The first term vanishes by the
short-range attenuation condition:
\[ \lim_{\varepsilon \to 0}  \int_{|z| < 1} |z|^{sp} \rho_{\varepsilon} (z)
   \hspace{0.17em} \mathrm{d} z = 0. \]
The second term converges to 1 by the mass escape condition:
\[ \lim_{\varepsilon \to 0}  \int_{|z| \geqslant 1} \rho_{\varepsilon} (z)
   \hspace{0.17em} \mathrm{d} z = 1. \]
Summing the limits, we get that the family of kernels belongs to
$\mathcal{A}_{s, p}$.
\end{proof}

\begin{remark}[The reverse implication does not hold]
A family of kernels may be admissible ($\in \mathcal{A}_{s,p}$) without fulfilling the moment conditions. To illustrate this, it is enough to consider a kernel family whose mass is concentrated at the origin.
  Precisely, let $\phi \geqslant 0$ be a $C^{\infty}_c (\R^N)$ function with
  $\mathrm{supp}\, \phi \subset B (0, 1)$ and normalized such that
  \[ \int_{B (0, 1)} |w|^{sp} \phi (w) \hspace{0.17em} \mathrm{d} w = 1. \]
  For every $\varepsilon > 0$, define $\rho_{\varepsilon} (z) :=
  \varepsilon^{- N - sp} \phi (z / \varepsilon)$. The support of this kernel
  is $\mathrm{supp}\, \rho_{\varepsilon} \subseteq B (0, \varepsilon)$.
     By a change of variables $z = \varepsilon w$, we can verify the
  admissibility condition:
  \begin{align*}
    \lim_{\varepsilon \to 0}  \int_{\R^N} (1 \wedge |z|^{sp})
    \rho_{\varepsilon} (z) \hspace{0.17em} \mathrm{d} z & = \lim_{\varepsilon \to 0} 
    \int_{|z| < \varepsilon} |z|^{sp} \varepsilon^{- N - sp} \phi \left(
    \frac{z}{\varepsilon} \right) \hspace{0.17em} \mathrm{d} z\\
    & = \lim_{\varepsilon \to 0}  \int_{|w| < 1} | \varepsilon w|^{sp}
    \varepsilon^{- N - sp} \phi (w) (\varepsilon^N \hspace{0.17em} \mathrm{d} w)\\
    & = \int_{|w| < 1} |w|^{sp} \phi (w) \hspace{0.17em} \mathrm{d} w = 1.
  \end{align*}
  So, $\{\rho_{\varepsilon} \}$ belongs to $\mathcal{A}_{s, p}$. However, the
  uniform moment conditions fail. For any fixed $R > 0$, we can choose
  $\varepsilon < R$. Then the support of $\rho_{\varepsilon}$ is entirely
  contained within $B (0, R)$. This means:
  \[ \int_{|z| > R} \rho_{\varepsilon} (z) \hspace{0.17em} \mathrm{d} z = 0 \quad \text{for
     all } \varepsilon < R. \]
  The limit is 0, which violates the mass-escape condition ($\lim \to 1$).
\end{remark}

On the other hand, if $\{ \rho_{\varepsilon} \}_{\varepsilon > 0}$ is in the
class $\mathcal{A}_{s, p}$, then the uniform moment conditions become
equivalent.

\begin{lemma}\label{lemma:equivconds}
  Let $\{ \rho_{\varepsilon} \}$ be a family of non-negative measurable
  kernels in $\R^N$. If the family belongs to the class 
  $\mathcal{A}_{s, p}$ then the uniform moment conditions \eqref{eq:massescape} and \eqref{eq:sratten} are equivalent.
\end{lemma}

\begin{proof}
  By assumption, the family of kernels is admissible, i.e.,
  \begin{equation}
    1 = \lim_{\varepsilon \rightarrow 0} \int_{\R^N} (1 \wedge | z |^{sp})
    \rho_{\varepsilon} (z) \hspace{0.17em} \mathrm{d} z. \label{eq:adcond1}
  \end{equation}
  (\eqref{eq:massescape}  $\Rightarrow$ \ref{eq:sratten}). The Mass Escape condition
  \eqref{eq:massescape} implies that for any $0 < R_1 < R_2$ we have
  \begin{equation}
    \lim_{\varepsilon \rightarrow 0} \int_{R_1 < |z| < R_2} \rho_{\varepsilon}
    (z) \hspace{0.17em} \mathrm{d} z = \lim_{\varepsilon \rightarrow 0} \int_{| z | > R_1}
    \rho_{\varepsilon} (z) - \lim_{\varepsilon \rightarrow 0}  \int_{| z | >
    R_2} \rho_{\varepsilon} (z) \hspace{0.17em} \mathrm{d} z = 0. \label{eq:adcond1.5}
  \end{equation}
  In particular,
  \begin{equation}
    \lim_{\varepsilon \rightarrow 0} \int_{1 < |z| < R} |z|^{sp}
    \rho_{\varepsilon} (z) \hspace{0.17em} \mathrm{d} z \leqslant R^{s p} \lim_{\varepsilon
    \rightarrow 0} \int_{1 < |z| < R} \rho_{\varepsilon} (z) \hspace{0.17em} \mathrm{d} z = 0.
    \label{eq:adcond2}
  \end{equation}
  Now, observe that it is sufficient to prove that \eqref{eq:sratten} holds for
  every $R > 1$. But for $R > 1$ we have
  \begin{eqnarray}
    \int_{|z| < R} |z|^{sp} \rho_{\varepsilon} (z) \hspace{0.17em} \mathrm{d} z & = &
    \int_{|z| < 1} (1 \wedge | z |^{sp}) \rho_{\varepsilon} (z) \hspace{0.17em} \mathrm{d} z +
    \int_{1 < |z| < R} |z|^{sp} \rho_{\varepsilon} (z) \hspace{0.17em} \mathrm{d} z
    \nonumber\\
    & = & \int_{\R^N} (1 \wedge | z |^{sp}) \rho_{\varepsilon} (z) \hspace{0.17em} \mathrm{d}
    z - \int_{|z| > 1} \rho_{\varepsilon} (z) \hspace{0.17em} \mathrm{d} z + \int_{1 < |z| <
    R} |z|^{sp} \rho_{\varepsilon} (z) \hspace{0.17em} \mathrm{d} z, \nonumber
  \end{eqnarray}
  from which, taking into account \eqref{eq:adcond1}, \eqref{eq:adcond2}, and
  condition \eqref{eq:massescape}  applied with $R = 1$, passing to the limit for
  $\varepsilon \rightarrow 0$ we get condition~\eqref{eq:sratten}.
  
  \smallskip
  
 \noindent (\eqref{eq:sratten} $\Rightarrow$ \eqref{eq:massescape}). The Short-Range Attenuation
  condition \eqref{eq:sratten} implies that for any $0 < R_1 < R_2$ we have
  \begin{equation}
    \lim_{\varepsilon \rightarrow 0} \int_{R_1 < |z| < R_2} |z|^{sp}
    \rho_{\varepsilon} (z) \hspace{0.17em} \mathrm{d}z = \lim_{\varepsilon \rightarrow 0}
    \int_{| z | < R_2} |z|^{sp} \rho_{\varepsilon} (z) - \lim_{\varepsilon
    \rightarrow 0}  \int_{| z | < R_1} |z|^{sp} \rho_{\varepsilon} (z)
    \hspace{0.17em} \mathrm{d} z = 0. \label{eq:adcond1.52}
  \end{equation}
  In particular, for any $R < 1$ we have
  \begin{equation}
    \lim_{\varepsilon \rightarrow 0} \int_{R < |z| < 1} \rho_{\varepsilon} (z)
   \hspace{0.17em} \mathrm{d} z \leqslant R^{- s p} \lim_{\varepsilon \rightarrow 0} \int_{R
    < |z| < 1}  | z |^{s p} \rho_{\varepsilon} (z) \hspace{0.17em} \mathrm{d} z = 0.
    \label{eq:adcond22}
  \end{equation}
  Now, observe that it is sufficient to prove that \eqref{eq:massescape}  holds for
  every $R < 1$. But for $R < 1$ we have
  \begin{eqnarray}
    \int_{|z| > R} \rho_{\varepsilon} (z) \hspace{0.17em} \mathrm{d} z & = & \int_{R < |z| <
    1} \rho_{\varepsilon} (z) \hspace{0.17em} \mathrm{d} z + \int_{|z| > 1} (1 \wedge | z
    |^{sp}) \rho_{\varepsilon} (z) \hspace{0.17em} \mathrm{d} z \nonumber\\
    & = & \int_{R < |z| < 1} \rho_{\varepsilon} (z) \hspace{0.17em} \mathrm{d} z +
    \int_{\R^N} (1 \wedge | z |^{sp}) \rho_{\varepsilon} (z) \hspace{0.17em} \mathrm{d} z -
    \int_{|z| < 1} | z |^{sp} \rho_{\varepsilon} (z) \hspace{0.17em} \mathrm{d} z, \nonumber
  \end{eqnarray}
  from which, taking into account \eqref{eq:adcond1}, \eqref{eq:adcond22}, and
  condition \eqref{eq:sratten} applied with $R = 1$, passing to the limit for
  $\varepsilon \rightarrow 0$ we get condition~\eqref{eq:massescape} .
\end{proof}

We can now state and prove the main result of this section.

\begin{theorem}[Conditional fractional Sobolev setting]
  \label{3-Thm2}Let $p \in [1, \infty)$ and $s \in (0, 1)$. Assume that
  $\{\rho_{\varepsilon} \} \subset \mathcal{A}_{s, p}$. The following three
  statements are equivalent:
  \begin{enumerate}
    \item  \label{scond:R-fix-Sob} \text{{\bfseries{(Uniform moment conditions)}}} For every fixed
    radius $R > 0$, as $\varepsilon \to 0$ the kernels satisfy
    \begin{equation}
      \label{eq:kernels-R_fix-s} \lim_{\varepsilon \to 0}  \int_{|z| > R}
      \rho_{\varepsilon} (z) \hspace{0.17em} \mathrm{d} z = 1 \quad \text{and}
      \quad \lim_{\varepsilon \to 0}  \int_{|z| < R} |z|^{sp}
      \rho_{\varepsilon} (z) \hspace{0.17em} \mathrm{d} z = 0.
    \end{equation}
    \item \label{cond:R-lim-s} \text{{\bfseries{(Iterated-limits conditions)}}} The iterated limits
    taken in the order ``first $\varepsilon \to 0$, then $R \to \infty$''
    satisfy
    \begin{equation}
      \label{eq:kernels-R_lim-s} \lim_{R \to \infty} \lim_{\varepsilon \to 0} 
      \int_{|z| > R} \rho_{\varepsilon} (z) \hspace{0.17em} \mathrm{d} z = 1
      \quad \text{and} \quad \lim_{R \to \infty} \lim_{\varepsilon \to 0} 
      \int_{|z| < R} |z|^{sp} \rho_{\varepsilon} (z) \hspace{0.17em}
      \mathrm{d} z = 0.
    \end{equation}
    \item\label{seq:MS+II-s}  \text{{\bfseries{(Maz'ya--Shaposhnikova formula in $W^{s, p}
    (\mathbb{R}^N)$)}}} For every $u \in W^{s, p} (\mathbb{R}^N)$ the nonlocal
    energy $\mathcal{F}_{\varepsilon} (u)$ is well-defined for all
    sufficiently small $\varepsilon$, and
    \begin{equation}
      \label{seq:MS-s} \lim_{\varepsilon \to 0} \mathcal{F}_{\varepsilon} (u)
      = 2 \|u\|^p_{L^p (\mathbb{R}^N)} .
    \end{equation}
    Moreover, for each fixed $R > 0$,
    \begin{equation}
      \label{eq:II-s} \lim_{\varepsilon \to 0}  \int_{|z| < R}
      \int_{\mathbb{R}^N} \rho_{\varepsilon} (z)  \hspace{0.17em} |u (x + z) -
      u (x) |^p \hspace{0.17em} \mathrm{d} x \hspace{0.17em} \mathrm{d} z = 0.
    \end{equation}
  \end{enumerate}
\end{theorem}

\noindent\textbf{Proof\ }The implication (\ref{scond:R-fix-Sob} $\Rightarrow$
\ref{cond:R-lim-s}) follows in the same way as the analogous implication in
Theorem \ref{thm:MS-smooth}, while the implication (\ref{cond:R-lim-s}
$\Rightarrow$ \ref{seq:MS+II-s}) is an immediate consequence of Theorem
\ref{thm:MS-frac}. Thus, it suffices to prove the converse implication
(\ref{seq:MS+II-s} $\Rightarrow$ \ref{scond:R-fix-Sob}).

Assume that \eqref{seq:MS-s} and \eqref{eq:II-s} hold. Since $C_c^{\infty}
(\mathbb{R}^N) \subset W^{s, p} (\mathbb{R}^N)$, by
Theorem~\ref{thm:MS-smooth}, the first relation in \eqref{eq:kernels-R_fix-s}
already holds; that is, for every $R > 0$ one has
\begin{equation}
  \label{eq:mass_at_infty} \lim_{\varepsilon \to 0}  \int_{|z| > R}
  \rho_{\varepsilon} (z) \hspace{0.17em} \mathrm{d} z = 1.
\end{equation}
But then, Lemma~\ref{lemma:equivconds}
completes the proof of the implication \ (\ref{seq:MS+II-s} $\Rightarrow$
\ref{scond:R-fix-Sob}), and hence of the
theorem.\hspace*{\fill}$\Box$\medskip

\section*{Acknowledgments}
This project was initiated during a visit of A.P.\ at the Institut für Analysis und Scientific Computing, TU Wien, whose excellent working conditions are warmly acknowledged.  

\noindent The research of E.D.\ was supported by the Austrian Science Fund (FWF) through projects \href{https://www.doi.org/10.55776/F65}{10.55776/F65}, \href{https://www.doi.org/10.55776/Y1292}{10.55776/Y1292}, \href{https://www.doi.org/10.55776/P35359}{10.55776/P35359}, and \href{https://www.doi.org/10.55776/F100800}{10.55776/F100800}.  

\noindent The work of \textsc{G.~Di F.} was partially supported by the Italian Ministry of Education and Research through the PRIN2022 project \emph{Variational Analysis of Complex Systems in Material Science, Physics and Biology} (No.~2022HKBF5C).  

\noindent The research of R.G.\ was funded by the Austrian Science Fund (FWF) through project \href{https://www.doi.org/10.55776/P34609}{10.55776/P34609}. R.G.\ gratefully acknowledges the support and hospitality of TU Wien and MedUni Wien.  

\noindent The research of A.P.\ was partially supported by the GNAMPA project \emph{Structures of sub-Riemannian hypersurfaces in Heisenberg groups}. Further support was provided by MIUR and the University of Trento (Italy) through the PRIN project \emph{Regularity problems in Sub-Riemannian structures} (MUR-PRIN 2022, Project code: 2022F4F2LH). For open-access purposes, the authors have applied a CC BY public copyright license to any author-accepted manuscript version arising from this submission.  

\noindent \textsc{G.~Di F.} and A.P.\ are members of the \emph{Gruppo Nazionale per l’Analisi Matematica, la Probabilità e le loro Applicazioni} (GNAMPA), which is part of the \emph{Istituto Nazionale di Alta Matematica} (INdAM).

\medskip
\noindent \small{\textbf{Conflicts of Interest}. On behalf of all authors, the corresponding author states that there is no conflict of interest. { Data sharing not applicable to this article as no datasets were generated or analyzed during the current study}


\bibliographystyle{siam}
\bibliography{Bib}

\def\cprime{$'$}
\begin{thebibliography}{10}

\bibitem{AT-T}
{\sc L.~Ambrosio and P.~Tilli}, {\em Topics on analysis in metric space},
  Oxford University Press, 2004.

\bibitem{Applebaum2009}
{\sc D.~Applebaum}, {\em Lévy Processes and Stochastic Calculus}, Cambridge
  Studies in Advanced Mathematics, Cambridge University Press, Cambridge, UK,
  2nd~ed., 2009.
\newblock Standard reference on Lévy processes (Lévy measures / long-tailed
  kernels).

\bibitem{MR4507709}
{\sc A.~Arroyo-Rabasa and P.~Bonicatto}, {\em A {B}ourgain-{B}rezis-{M}ironescu
  representation for functions with bounded deformation}, Calc. Var. Partial
  Differential Equations, 62 (2023), pp.~Paper No. 33, 22.

\bibitem{BBM}
{\sc J.~Bourgain, H.~Brezis, and P.~Mironescu}, {\em Another look at {S}obolev
  spaces}, in Optimal control and partial differential equations, IOS,
  Amsterdam, 2001, pp.~439--455.

\bibitem{MR4525722}
{\sc D.~Brazke, A.~Schikorra, and P.-L. Yung}, {\em
  Bourgain-{B}rezis-{M}ironescu convergence via {T}riebel-{L}izorkin spaces},
  Calc. Var. Partial Differential Equations, 62 (2023), pp.~Paper No. 41, 33.

\bibitem{MR3556344}
{\sc H.~Brezis and H.-M. Nguyen}, {\em The {BBM} formula revisited}, Atti
  Accad. Naz. Lincei Rend. Lincei Mat. Appl., 27 (2016), pp.~515--533.

\bibitem{MR3485124}
\leavevmode\vrule height 2pt depth -1.6pt width 23pt, {\em Two subtle convex
  nonlocal approximations of the {BV}-norm}, Nonlinear Anal., 137 (2016),
  pp.~222--245.

\bibitem{MR3749763}
\leavevmode\vrule height 2pt depth -1.6pt width 23pt, {\em Non-local
  functionals related to the total variation and connections with image
  processing}, Ann. PDE, 4 (2018), pp.~Paper No. 9, 77.

\bibitem{MR4011115}
\leavevmode\vrule height 2pt depth -1.6pt width 23pt, {\em Non-local,
  non-convex functionals converging to {S}obolev norms}, Nonlinear Anal., 191
  (2020), pp.~111626, 9.

\bibitem{MR4482042}
{\sc H.~Brezis, A.~Seeger, J.~Van~Schaftingen, and P.-L. Yung}, {\em Sobolev
  spaces revisited}, Atti Accad. Naz. Lincei Rend. Lincei Mat. Appl., 33
  (2022), pp.~413--437.

\bibitem{MR4275122}
{\sc H.~Brezis, J.~Van~Schaftingen, and P.-L. Yung}, {\em A surprising formula
  for {S}obolev norms}, Proc. Natl. Acad. Sci. USA, 118 (2021), pp.~Paper No.
  e2025254118, 6.

\bibitem{MR1942116}
{\sc K.~Brezis}, {\em How to recognize constant functions. {A} connection with
  {S}obolev spaces}, Uspekhi Mat. Nauk, 57 (2002), pp.~59--74.

\bibitem{MR4449863}
{\sc E.~Bru\`e, M.~Calzi, G.~E. Comi, and G.~Stefani}, {\em A distributional
  approach to fractional {S}obolev spaces and fractional variation: asymptotics
  {II}}, C. R. Math. Acad. Sci. Paris, 360 (2022), pp.~589--626.

\bibitem{Buades}
{\sc A.~Buades, B.~Coll, and J.-M. Morel}, {\em A non-local algorithm for image
  denoising}, in 2005 IEEE Computer Society Conference on Computer Vision and
  Pattern Recognition (CVPR’05), vol.~2, IEEE, 2005, pp.~60--65.

\bibitem{MR4453966}
{\sc F.~Buseghin, N.~Garofalo, and G.~Tralli}, {\em On the limiting behaviour
  of some nonlocal seminorms: a new phenomenon}, Ann. Sc. Norm. Super. Pisa Cl.
  Sci. (5), 23 (2022), pp.~837--875.

\bibitem{Cri}
{\sc V.~Crismale, L.~De~Luca, A.~Kubin, A.~Ninno, and M.~Ponsiglione}, {\em The
  variational approach to {$s$}-fractional heat flows and the limit cases
  {$s\to0^+$} and {$s\to1^-$}}, J. Funct. Anal., 284 (2023), pp.~Paper No.
  109851, 38.

\bibitem{DDFG24}
{\sc E.~Davoli, G.~Di~Fratta, and R.~Giorgio}, {\em A
  {B}ourgain--{B}rezis--{M}ironescu formula accounting for nonlocal
  antisymmetric exchange interactions}, SIAM Journal on Mathematical Analysis,
  56 (2024), pp.~6995--7013.

\bibitem{Davoli}
{\sc E.~Davoli, G.~Di~Fratta, and V.~Pagliari}, {\em Sharp conditions for the
  validity of the {B}ourgain--{B}rezis--{M}ironescu formula}, Proceedings of
  the Royal Society of Edinburgh Section A: Mathematics,  (2024), pp.~1--24.

\bibitem{DiFio2020bmo}
{\sc G.~Di~Fratta and A.~Fiorenza}, {\em {B}{M}{O}-type seminorms from
  {E}scher-type tessellations}, Journal of Functional Analysis, 279 (2020),
  p.~108556.

\bibitem{DFGL25}
{\sc G.~Di~Fratta, R.~Giorgio, and L.~Lombardini}, {\em Nonlocal
  micromagnetics: Compactness criteria, existence of minimizers, and {B}rown's
  fundamental theorem}, Preprint arXiv:2502.05532,  (2025).

\bibitem{DMS19}
{\sc S.~Di~Marino and M.~Squassina}, {\em New characterizations of {S}obolev
  metric spaces}, J. Funct. Anal., 276 (2019), pp.~1853--1874.

\bibitem{Nezza2012}
{\sc E.~{Di~Nezza}, G.~Palatucci, and E.~Valdinoci}, {\em Hitchhiker's guide to
  the fractional sobolev spaces}, Bulletin des Sciences Mathématiques, 136
  (2012), pp.~521--573.

\bibitem{MR3007726}
{\sc S.~Dipierro, A.~Figalli, G.~Palatucci, and E.~Valdinoci}, {\em Asymptotics
  of the {$s$}-perimeter as {$s\searrow0$}}, Discrete Contin. Dyn. Syst., 33
  (2013), pp.~2777--2790.

\bibitem{Du2012}
{\sc Q.~Du, M.~Gunzburger, R.~B. Lehoucq, and K.~Zhou}, {\em Analysis and
  approximation of nonlocal diffusion problems with volume constraints}, SIAM
  Review, 54 (2012), pp.~667--696.

\bibitem{Fog1}
{\sc G.~Foghem}, {\em $L^2$-Theory for nonlocal operators on domains}, PhD
  thesis, Bielefeld University, 2020.

\bibitem{Fog}
\leavevmode\vrule height 2pt depth -1.6pt width 23pt, {\em Stability of
  complement value problems for $p$-{L}{\'e}vy operators}, Nonlinear
  Differential Equations and Applications NoDEA, 32 (2024), p.~1.

\bibitem{MR4685023}
{\sc N.~Garofalo and G.~Tralli}, {\em A
  {B}ourgain-{B}rezis-{M}ironescu-{D}\'{a}vila theorem in {C}arnot groups of
  step two}, Comm. Anal. Geom., 31 (2023), pp.~32--341.

\bibitem{Stefani}
{\sc L.~Gennaioli and G.~Stefani}, {\em Sharp conditions for the {BBM} formula
  and asymptotics of heat content-type energies}.
\newblock Preprint, arXiv:2502.14655, 2025.

\bibitem{Gilboa2009}
{\sc G.~Gilboa and S.~Osher}, {\em Nonlocal operators with applications to
  image processing}, Multiscale Modeling \& Simulation, 7 (2009),
  pp.~1005--1028.

\bibitem{MR4375837}
{\sc W.~G\'{o}rny}, {\em Bourgain-{B}rezis-{M}ironescu approach in metric
  spaces with {E}uclidean tangents}, J. Geom. Anal., 32 (2022), pp.~Paper No.
  128, 22.

\bibitem{MR4782147}
{\sc B.-X. Han}, {\em On the asymptotic behaviour of the fractional {S}obolev
  seminorms: a geometric approach}, J. Funct. Anal., 287 (2024), pp.~Paper No.
  110608, 25.

\bibitem{MSPin}
{\sc B.~X. Han, A.~Pinamonti, Z.~Xu, and K.~Zambanini}, {\em
  {M}az’ya–{S}haposhnikova {M}eet {B}ishop–{G}romov}, Potential Anal,
  (2024).
\newblock 10.1007/s11118-024-10179-9.

\bibitem{MR4788002}
{\sc P.~Lahti, A.~Pinamonti, and X.~Zhou}, {\em B{V} functions and nonlocal
  functionals in metric measure spaces}, J. Geom. Anal., 34 (2024), pp.~Paper
  No. 318, 34.

\bibitem{LPZ1}
{\sc P.~Lahti, A.~Pinamonti, and X.~Zhou}, {\em A characterization of {BV} and
  {S}obolev functions via nonlocal functionals in metric spaces}, Nonlinear
  Anal., 241 (2024), p.~113467.

\bibitem{MR3986928}
{\sc N.~Lam, A.~Maalaoui, and A.~Pinamonti}, {\em Characterizations of
  anisotropic high order {S}obolev spaces}, Asymptot. Anal., 113 (2019),
  pp.~239--260.

\bibitem{L23}
{\sc G.~Leoni}, {\em A first course in fractional Sobolev spaces}, vol.~229,
  American Mathematical Society, 2023.

\bibitem{MR2832587}
{\sc G.~Leoni and D.~Spector}, {\em Characterization of {S}obolev and {$BV$}
  spaces}, J. Funct. Anal., 261 (2011), pp.~2926--2958.

\bibitem{MR3132740}
\leavevmode\vrule height 2pt depth -1.6pt width 23pt, {\em Corrigendum to
  ``{C}haracterization of {S}obolev and {$BV$} spaces'' [{J}. {F}unct. {A}nal.
  261 (10) (2011) 2926--2958]}, J. Funct. Anal., 266 (2014), pp.~1106--1114.

\bibitem{MR3886626}
{\sc A.~Maalaoui and A.~Pinamonti}, {\em Interpolations and fractional
  {S}obolev spaces in {C}arnot groups}, Nonlinear Anal., 179 (2019),
  pp.~91--104.

\bibitem{Maleki2013}
{\sc A.~Maleki, M.~Narayan, and R.~G. Baraniuk}, {\em Anisotropic nonlocal
  means denoising}, Applied and Computational Harmonic Analysis, 35 (2013),
  pp.~452--482.
\newblock arXiv:1112.0311 (v2) available.

\bibitem{MS}
{\sc V.~Maz'ya and T.~Shaposhnikova}, {\em On the {B}ourgain, {B}rezis, and
  {M}ironescu theorem concerning limiting embeddings of fractional {S}obolev
  spaces}, J. Funct. Anal., 195 (2002), pp.~230--238.

\bibitem{MR3299669}
{\sc V.~Munnier}, {\em Integral energy characterization of {H}aj\l
  asz-{S}obolev spaces}, J. Math. Anal. Appl., 425 (2015), pp.~381--406.

\bibitem{Ngu}
{\sc H.-M. Nguyen, A.~Pinamonti, M.~Squassina, and E.~Vecchi}, {\em Some
  characterizations of magnetic {S}obolev spaces}, Complex Var. Elliptic Equ.,
  65 (2020), pp.~1104--1114.

\bibitem{MR3601583}
{\sc A.~Pinamonti, M.~Squassina, and E.~Vecchi}, {\em The
  {M}az'ya-{S}haposhnikova limit in the magnetic setting}, J. Math. Anal.
  Appl., 449 (2017), pp.~1152--1159.

\bibitem{MR3975602}
\leavevmode\vrule height 2pt depth -1.6pt width 23pt, {\em Magnetic
  {BV}-functions and the {B}ourgain-{B}rezis-{M}ironescu formula}, Adv. Calc.
  Var., 12 (2019), pp.~225--252.

\bibitem{Ponce}
{\sc A.~C. Ponce}, {\em A new approach to {S}obolev spaces and connections to
  {$\Gamma$}-convergence}, Calc. Var. Partial Differential Equations, 19
  (2004), pp.~229--255.

\bibitem{Silling2000}
{\sc S.~Silling}, {\em Reformulation of elasticity theory for discontinuities
  and long-range forces}, Journal of the Mechanics and Physics of Solids, 48
  (2000), pp.~175--209.
\newblock Peridynamics -- original formulation.

\end{thebibliography}

\end{document}